\documentclass{amsart}
\usepackage{amssymb}
\usepackage[all]{xy}

\newtheorem{thm}{Theorem}[section]
\newtheorem{lem}[thm]{Lemma}
\newtheorem{prop}[thm]{Proposition}
\newtheorem{cor}[thm]{Corollary}
\newtheorem{conj}[thm]{Conjecture}

\theoremstyle{definition}
\newtheorem{defn}[thm]{Definition}

\theoremstyle{remark}
\newtheorem{rmk}[thm]{Remark}
\newtheorem{exam}[thm]{Example}

\numberwithin{equation}{section}

\newcommand{\IC}{\mathrm{IC}}

\newcommand{\C}{\mathbb{C}}
\newcommand{\F}{\mathbb{F}}
\newcommand{\N}{\mathbb{N}}
\newcommand{\Oo}{\mathbb{O}}
\newcommand{\Tt}{\mathbb{T}}
\newcommand{\Ss}{\mathbb{S}}
\newcommand{\Ee}{\mathbb{E}}

\newcommand{\Z}{\mathbb{Z}}

\newcommand{\cE}{\mathcal{E}}

\newcommand{\cN}{\mathcal{N}}
\newcommand{\cO}{\mathcal{O}}
\newcommand{\cP}{\mathcal{P}}
\newcommand{\cQ}{\mathcal{Q}}
\newcommand{\cS}{\mathcal{S}}
\newcommand{\cT}{\mathcal{T}}

\newcommand{\fN}{\mathfrak{N}}

\newcommand{\fo}{\mathfrak{o}}
\newcommand{\fsp}{\mathfrak{sp}}
\newcommand{\fgl}{\mathfrak{gl}}

\newcommand{\tPsi}{\widetilde{\Psi}}

\newcommand{\hPhi}{\widehat{\Phi}}
\newcommand{\tchi}{\widetilde{\chi}}
\newcommand{\tV}{\widetilde{V}}
\newcommand{\tQ}{\widetilde{Q}}

\newcommand{\dup}{{\mathrm{dup}}}
\newcommand{\bt}{{\mathbf{t}}}

\newcommand{\isomto}{\overset{\sim}{\rightarrow}}

\newcommand{\disjunion}{\cup}
\newcommand{\bigdisjunion}{\bigcup}

\DeclareMathOperator{\End}{End}

\DeclareMathOperator{\im}{im}

\title[Pieces of nilpotent cones for classical groups]
{Pieces of nilpotent cones for classical groups}
\author{Pramod N. Achar}
\address{Department of Mathematics\\
Louisiana State University\\
Baton Rouge, LA 70803-4918\\
USA}
\email{pramod@math.lsu.edu}
\author{Anthony Henderson}
\address{School of Mathematics and Statistics\\
University of Sydney, NSW 2006\\
Australia}
\email{anthonyh@maths.usyd.edu.au}
\author{Eric Sommers}
\address{Department of Mathematics and Statistics\\
University of Massachusetts\\
Amherst, MA 01003-4515\\
USA}
\email{esommers@math.umass.edu}

\thanks{The first author's research was supported by Louisiana
Board of Regents grant NSF(2008)-LINK-35 and by National Security
Agency grant H98230-09-1-0024. The second
author's research was supported by
Australian Research Council grant DP0985184.}


\begin{document}

\begin{abstract}
We compare orbits in the nilpotent cone of type $B_n$, that of type $C_n$, and Kato's
exotic nilpotent cone. We prove that the number of $\F_q$-points
in each nilpotent orbit of type $B_n$ or $C_n$ equals that in a corresponding union of
orbits, called a type-$B$ or type-$C$ piece, in the exotic nilpotent cone. 
This is a finer version of Lusztig's result that
corresponding special pieces in types $B_n$ and $C_n$ have the same number of 
$\F_q$-points. The proof requires studying the case of characteristic $2$, where more
direct connections between the three nilpotent cones can be established.
We also prove that the type-$B$ and type-$C$ pieces of the exotic nilpotent cone are smooth in any characteristic.
\end{abstract}

\maketitle

\section{Introduction}
Let $\F$ be an algebraically closed field; for the moment, assume that the 
characteristic of $\F$ is not $2$.
The algebraic groups $SO_{2n+1}(\F)$ (of type $B_n$) and $Sp_{2n}(\F)$ (of type $C_n$)
share many features by virtue of having dual root data and hence isomorphic Weyl groups:
$W(B_n) \cong W(C_n) \cong \{\pm 1\}\wr S_n$.


The connection between the orbits in their respective nilpotent cones 
$\cN(\fo_{2n+1})$ and $\cN(\fsp_{2n})$
is 
subtle. (Recall that in characteristic $\neq 2$, these nilpotent cones are isomorphic
to the unipotent varieties of the corresponding groups, so everything we say about nilpotent orbits implies an analogous statement for unipotent classes.)
The nilpotent orbits are most commonly parametrized by their Jordan types:
in the case of $\cN(\fo_{2n+1})$, these are the
partitions of $2n+1$ in which every even part has even multiplicity; in the case
of $\cN(\fsp_{2n})$, the partitions of $2n$ in which every odd part has
even multiplicity. But to make a connection between the two nilpotent cones, it
is more natural to use the parametrization derived from the Springer correspondence.

In the Springer correspondence, nilpotent orbits (or, strictly speaking,
the trivial local systems on such orbits)
correspond to a subset of the irreducible representations of the Weyl group.
The irreducible representations of $\{\pm 1\}\wr S_n$ are parametrized
by the set $\cQ_n$ of bipartitions of $n$, on which we put the partial order
of `interleaved dominance', as in \cite{spaltenstein} and \cite{ah}: 
$(\rho;\sigma)\leq(\mu;\nu)$ means that the composition $(\rho_1,\sigma_1,
\rho_2,\sigma_2,\cdots)$ is dominated by the composition $(\mu_1,\nu_1,
\mu_2,\nu_2,\cdots)$. 
In the conventions used in this paper (for which see Section 2),
the nilpotent orbits in $\cN(\fo_{2n+1})$
and $\cN(\fsp_{2n})$, each with the
partial order given by closure inclusion, correspond respectively
to the sub-posets of $\cQ_n$ consisting of
\emph{$B$-distinguished} and \emph{$C$-distinguished} bipartitions:
\[
\begin{split}
\cQ_n^B&=\{(\mu;\nu)\in\cQ_n\,|\,\mu_i\geq\nu_i-2,\,\nu_i\geq\mu_{i+1},
\text{ for all }i\},\text{ and}\\
\cQ_n^C&=\{(\mu;\nu)\in\cQ_n\,|\,\mu_i\geq\nu_i-1,\,\nu_i\geq\mu_{i+1}-1,
\text{ for all }i\}.
\end{split}
\] 

Therefore, when comparing the two nilpotent cones, one is led
to the sub-poset of \emph{special} bipartitions (which label Lusztig's special
representations of $\{\pm 1\}\wr S_n$):
\[
\cQ_n^\circ=\cQ_n^B\cap\cQ_n^C=\{(\mu;\nu)\in\cQ_n\,|\,\mu_i\geq\nu_i-1,
\,\nu_i\geq\mu_{i+1},\text{ for all }i\}.
\]
Spaltenstein \cite{spaltenstein1}
observed that for every $(\rho;\sigma)\in\cQ_n^B$, there is a unique
$(\mu;\nu)\in\cQ_n^\circ$ which satisfies $(\rho;\sigma)\leq(\mu;\nu)$
and is minimal among all special bipartitions with this property; hence one
can group the orbits in $\cN(\fo_{2n+1})$ together into locally closed
subvarieties called \emph{special pieces}, one for each $(\mu;\nu)\in\cQ_n^\circ$.
The same is true for $\cQ_n^C$, so one can also define special pieces in
$\cN(\fsp_{2n})$ with the same parameter set.
In fact, special pieces can be defined in the nilpotent 
cones of all types -- see Spaltenstein \cite{spaltenstein1} and Lusztig \cite{lusztig:notes}.

The special pieces are a natural
basis of comparison between the
nilpotent cones 
in types $B_n$ and $C_n$. 
In \cite[Theorem 0.9, 6.9]{lusztig:notes}, Lusztig proved that
for any $(\mu;\nu)\in\cQ_n^\circ$, the corresponding special pieces of 
$\cN(\fo_{2n+1})$ and $\cN(\fsp_{2n})$ have the same number of
$\F_q$-points for every finite subfield $\F_q$; if $\F=\C$, the statement
becomes that the corresponding special pieces have the same equivariant Betti numbers. Lusztig's proof 
consists of computing each side of the equality independently, and showing
that the answers are the same; it does not give any direct connection between
the special pieces.

More recently, the work of Syu Kato \cite{kato:exotic,kato:deformations} has 
revealed that
there is a third nilpotent cone which in some respects is intermediate between
type $B$ and type $C$. This is the \emph{exotic nilpotent cone}
\[
\fN=\F^{2n}\times\{x\in\fgl_{2n}\,|\,x\text{ nilpotent},
\langle xv,v\rangle=0\text{ for all }v\in\F^{2n}\},
\]
where $\langle\cdot,\cdot\rangle$ denotes a nondegenerate alternating form. 
This exotic nilpotent
cone may seem more closely affiliated with type $C_n$ in that it is the 
symplectic group
$Sp_{2n}$ which acts on it, but note that it is the Hilbert null-cone in the
reducible representation $\F^{2n}\oplus\Lambda^2(\F^{2n})$, whose weights
are the same as those of $\fo_{2n+1}$ (under the usual identification of
the weight lattice of $Sp_{2n}$ with the root lattice of $SO_{2n+1}$).

The orbits in $\fN$, with the partial order given by
closure inclusion, correspond to the entire poset $\cQ_n$, as proved in 
\cite{kato:exotic} and \cite[Section 6]{ah}. So in $\fN$ one can not only
define special pieces in the same way as for $\cN(\fo_{2n+1})$ and 
$\cN(\fsp_{2n})$, one can define \emph{type-$B$ pieces} parametrized by
$\cQ_n^B$ and \emph{type-$C$ pieces} parametrized by $\cQ_n^C$. 

In the main result of this paper, Theorem \ref{thm:main}, we
prove equalities of numbers of $\F_q$-points between
the type-$B$ pieces of $\fN$ and the orbits of $\cN(\fo_{2n+1})$, and between
the type-$C$ pieces of $\fN$ and the orbits of $\cN(\fsp_{2n})$.
These finer equalities automatically imply that the number of $\F_q$-points
in a special piece of $\fN$ equals that in the corresponding special pieces
of $\cN(\fo_{2n+1})$ and $\cN(\fsp_{2n})$; thus we have a new proof of
Lusztig's result, via the exotic nilpotent cone (see Section 2 for the details).

The key to proving Theorem \ref{thm:main} is to remove the restriction on the
characteristic of $\F$, considering also the case of the
`bad' characteristic $2$. The classification of nilpotent orbits in characteristic
$2$ is slightly more complicated, but Lusztig and Xue have shown in 
\cite[Appendix]{lusztig:unip}
that there is a uniform way to define 
certain
pieces of $\cN(\fo_{2n+1})$ and $\cN(\fsp_{2n})$,
which in characteristic $\neq 2$ are the 
nilpotent orbits, and which in characteristic $2$ are unions
of nilpotent orbits. The number of $\F_q$-points in 
one of these {\it nilpotent pieces} is given by a polynomial in $q$ which
is independent of the characteristic. For the exotic nilpotent cone, 
the classification of orbits and the polynomials giving the number of
$\F_q$-points are the same in all characteristics. Since it suffices to prove an equality
of polynomials in $q$ in the case that $q$ is a power of $2$, 
we can restrict our attention purely to the case of characteristic $2$, where it is
possible to make direct connections between the three nilpotent cones.

It is perhaps not surprising that it should be easier to relate types $B_n$ and $C_n$ in 
characteristic $2$, since that is the characteristic in which there is an 
isogeny $SO_{2n+1}\to Sp_{2n}$. But this isogeny
does not induce a Lie algebra isomorphism, and the connection
between the orbits in $\cN(\fo_{2n+1})$ and $\cN(\fsp_{2n})$ only becomes apparent when one
uses the exotic nilpotent cone $\fN$ as a bridge between them.

One half of the bridge was constructed in
\cite{kato:deformations}, where Kato observed that there is an $Sp_{2n}$-equivariant
bijective morphism (not an isomorphism) from $\fN$ to $\cN(\fsp_{2n})$. 
In Section 3, we deduce from this the corresponding half of Theorem \ref{thm:main}.

In Section 4, we supply
the second half of the bridge, observing that in characteristic $2$
there is an isomorphism from $\cN(\fo_{2n+1})$ to $\fN$. This isomorphism is not
$SO_{2n+1}$-equivariant, but we are still able to use it to show
that the number of $\F_q$-points in an orbit in $\cN(\fo_{2n+1})$ equals
that in a corresponding union of orbits in $\fN$ (see Theorem \ref{thm:poly-B2}).
This implies the other half of Theorem \ref{thm:main}.

In Section 5, we prove that the type-$B$ and type-$C$ pieces of $\fN$ (in any characteristic) are smooth, like the nilpotent pieces of
\cite[Appendix]{lusztig:unip}. We conjecture that the special pieces of $\fN$ are
rationally smooth, like those of $\cN(\fo_{2n+1})$ and $\cN(\fsp_{2n})$ in good
characteristic. We also make
some remarks about the problem of computing intersection cohomology of nilpotent orbit
closures in bad characteristic.

For the convenience of the reader, we list in Table~\ref{tab:intropieces} our notations for the various kinds of pieces considered in this paper, with the reference number where 
they are first defined. The entries of the first column indicate which nilpotent cone in
which characteristic contains the pieces listed in that row; the heading of each other
column indicates the subset of $\cQ_n$ to which the labels $(\mu;\nu)$ in that column belong (the subset
$\cQ_n^{B,2}$, not previously mentioned, will be defined before Theorem \ref{thm:spalt-B}).
Notice that pieces of the exotic nilpotent cone are designated
by blackboard-bold letters $\Oo$ (for orbits), $\Tt$ (for type-$B$ and type-$C$ pieces),
and $\Ss$ (for special pieces), while pieces of $\cN(\fo_{2n+1})$ and $\cN(\fsp_{2n})$
are designated by calligraphic letters $\cO$ (for orbits), $\cT$ (for nilpotent pieces),
and $\cS$ (for special pieces). The guiding principle of this paper is that for each column
of Table~\ref{tab:intropieces}, the pieces listed in that column should be closely
related (though not necessarily isomorphic): in particular, we show that there is a 
uniform polynomial which counts the $\F_q$-points of each of them.

\begin{table}
\[
\begin{array}{l|l|l|l|l|l|}
\text{cone; \ char }\F & \cQ_n & \cQ_n^{B,2} & \cQ_n^B\qquad\strut & \cQ_n^C\qquad\strut & \cQ_n^\circ\qquad\strut \\
\hline
\cN(\fo_{2n+1}); \neq 2^{\strut} & & & \cO^B_{\mu;\nu}\;[\ref{thm:gerst-B}] & & 
\cS^B_{\mu;\nu}\;[\ref{defn:special}] \\
\cN(\fo_{2n+1}); 2 & & \cO^{B,2}_{\mu;\nu}\;[\ref{thm:spalt-B}] 
& \cT^B_{\mu;\nu}\;[\ref{defn:nilppiece-B}] & & 
\cS^B_{\mu;\nu}\;[\ref{defn:nilppiece-B}] \\
\cN(\fsp_{2n}); \neq 2 & & & & \cO^C_{\mu;\nu}\;[\ref{thm:gerst-C}] & 
\cS^C_{\mu;\nu}\;[\ref{defn:special}] \\
\cN(\fsp_{2n}); 2 & \cO^{C,2}_{\mu;\nu}\;[\ref{thm:spalt-C}] & 
\cT^{C;B,2}_{\mu;\nu}\;[\ref{rmk:relationship}] 
& 
& \cT^C_{\mu;\nu}\;[\ref{defn:nilppiece-C}] & \cS^C_{\mu;\nu}\;[\ref{defn:nilppiece-C}] \\
\fN; \text{any} & \Oo_{\mu;\nu}\;[\ref{thm:exotic}] & 
\Tt^{B,2}_{\mu;\nu}\;[\ref{defn:pieces-B2}] & \Tt^B_{\mu;\nu}\;[\ref{defn:pieces}] 
& \Tt^C_{\mu;\nu}\;[\ref{defn:pieces}] & \Ss_{\mu;\nu}\;[\ref{defn:special}]
\end{array}
\]
\caption{Pieces of nilpotent cones for classical groups.}\label{tab:intropieces}
\end{table}

\textbf{Acknowledgements.} The authors are indebted to the Isaac Newton Institute in
Cambridge, and to the organizers of the Algebraic Lie Theory program
there in 2009, during which this paper was conceived. We are also grateful to Ting Xue for helpful conversations, and for 
keeping us informed of her work \cite{xue:note} as it progressed.
\section{Combinatorics of bipartitions}
Throughout this section, $\F$ denotes an algebraically closed field of characteristic 
not equal to $2$. We will recall the parametrization of orbits
in $\cN(\fo_{2n+1})$, $\cN(\fsp_{2n})$, and the exotic nilpotent cone $\fN$, 
over $\F$.   We will then
state our main result, Theorem \ref{thm:main}, which relates 
various
pieces of the three cones,
and show how it implies Lusztig's result on special pieces.

The orbits in $\cN(\fo_{2n+1})$ and $\cN(\fsp_{2n})$ are usually parametrized by
certain partitions of $2n+1$ and $2n$ respectively. To make comparisons between
these two cones and the exotic nilpotent cone, we need to change to the
parameters given by the Springer correspondence, which for types $B$ and $C$ was
determined by Shoji in \cite{shoji:springer}. It is more convenient for us to use the
combinatorics of bipartitions,
rather than Lusztig's symbols as in \cite[13.3]{carter} and \cite[10.1]{cm}.
Much of the combinatorics we will describe is also
to be found in \cite{xue:note}.

Throughout the paper, $n$ denotes a nonnegative integer.
A \emph{composition} of $n$ is a sequence $\pi=(\pi_1,\pi_2,\pi_3,\cdots)$ of nonnegative
integers, almost all zero, such that $|\pi|=n$, where $|\pi|$ denotes
$\pi_1+\pi_2+\pi_3+\cdots$. The set of compositions of $n$ becomes a poset under the
partial order of \emph{dominance}:
\[
\pi\leq\lambda\Longleftrightarrow \pi_1+\cdots+\pi_k\leq\lambda_1+\cdots+\lambda_k,
\text{ for all }k.
\]

A composition $\pi$ is a \emph{partition} if $\pi_i\geq\pi_{i+1}$ for all $i$.
The poset of partitions of $n$ is denoted $\cP_n$. We use the usual shorthand notation
for partitions: for example, $(21^2)$ means $(2,1,1,0,0,\cdots)$ and $\varnothing$ means
$(0,0,0,\cdots)$.

A composition $\pi$ is a \emph{quasi-partition} if $\pi_i\geq\pi_{i+2}$ for all $i$.
It is clear that the set of quasi-partitions of $n$ is in bijection with the set
$\cQ_n$ of \emph{bipartitions} of $n$, that is, ordered 
pairs of partitions $(\rho; \sigma)$ with $|\rho| + |\sigma| = n$. Namely, the bipartition
$(\rho;\sigma)$ corresponds to the quasi-partition $(\rho_1,\sigma_1,\rho_2,\sigma_2,\cdots)$ obtained by interleaving the parts. Following \cite{spaltenstein}, 
\cite{shoji:limit}, and \cite{ah}, we use the partial order on $\cQ_n$ corresponding
to the dominance order on quasi-partitions: that is, 
$(\rho;\sigma)\leq(\mu;\nu)$ if and only if
\[
\begin{split}
\rho_1+\sigma_1+\rho_2+\sigma_2+\cdots+\rho_k+\sigma_k
&\leq \mu_1+\nu_1+\mu_2+\nu_2+\cdots+\mu_k+\nu_k,\text{ and}\\
\rho_1+\sigma_1+\cdots+\rho_k+\sigma_k+\rho_{k+1}
&\leq \mu_1+\nu_1+\cdots+\mu_k+\nu_k+\mu_{k+1},
\end{split}
\]
for all $k$.

It is well known that $\cQ_n$ parametrizes the irreducible characters of
$\{\pm 1\}\wr S_n\cong W(B_n)\cong W(C_n)$. We use the conventions of \cite[5.5]{gpf},
in which $\chi^{(n);\varnothing}$ is the trivial character, $\chi^{\varnothing;(n)}$
is the character whose kernel is $W(D_n)$, and tensoring with the sign
representation interchanges $\chi^{\mu;\nu}$ and $\chi^{\nu^\bt;\mu^\bt}$.


We also define, for any $(\mu;\nu)\in\cQ_n$,
\[
b(\mu;\nu)=0\mu_1+1\nu_1+2\mu_2+3\nu_2+4\mu_3+5\nu_3+\cdots.
\]
The representation-theoretic meaning of this quantity is explained 
in \cite[5.5.3]{gpf}: it is the smallest degree of the
coinvariant algebra of $W(B_n)\cong W(C_n)$ in which the character $\chi^{\mu;\nu}$ occurs.


\subsection{The nilpotent cone of type $C_n$}

Let $\cP_{2n}^C$ denote the sub-poset of $\cP_{2n}$
consisting of partitions of $2n$ in which every odd part has
even multiplicity.
We define a map $\Phi^C:\cQ_n\to\cP_{2n}^C$
as follows. For any $(\rho;\sigma)\in\cQ_n$, $\Phi^C(\rho;\sigma)$ is defined to be
the composition of $2n$ obtained from the quasi-partition
\begin{equation*}
(2\rho_1,2\sigma_1,2\rho_2,2\sigma_2,2\rho_3,2\sigma_3,\cdots)
\end{equation*}
by replacing
successive terms $2s,2t$ with $s+t,s+t$ whenever $2s<2t$. The quasi-partition
property implies that we never have to make
overlapping replacements, so this makes sense.

\begin{prop} \label{prop:phiC} Let $(\rho;\sigma),(\mu;\nu)\in\cQ_n$.
\begin{enumerate}
\item
$\Phi^C(\rho;\sigma)\in\cP_{2n}^C$.
\item
If $\lambda\in\cP_{2n}$, then $\lambda\geq
(2\rho_1,2\sigma_1,2\rho_2,2\sigma_2,\cdots)$ if and only if
$\lambda\geq\Phi^C(\rho;\sigma)$.
\item
If $(\rho;\sigma)\leq(\mu;\nu)$, then $\Phi^C(\rho;\sigma)\leq\Phi^C(\mu;\nu)$. 
\end{enumerate}
\end{prop}

\begin{proof}
Set $\pi=(2\rho_1,2\sigma_1,2\rho_2,2\sigma_2,\cdots)$, and for
brevity write $\xi$ for $\Phi^C(\rho;\sigma)$. It is clear from the
definition that any odd part of $\xi$ arises as a 
replacement of some pairs of successive parts of $\pi$, 
so its multiplicity must be even. To prove (1), we need only check that
$\xi_i\geq\xi_{i+1}$ for any $i$. If $\pi_i<\pi_{i+1}$, then $\xi_i=\xi_{i+1}$ 
by construction. On the other hand, if $\pi_i\geq\pi_{i+1}$, then 
\[
\xi_i=\min\left\{\pi_i,\frac{\pi_{i-1}+\pi_i}{2}\right\}
\geq\max\left\{\pi_{i+1},\frac{\pi_{i+1}+\pi_{i+2}}{2}\right\}
=\xi_{i+1},
\] 
since $\pi_{i-1}\geq\pi_{i+1}$ and $\pi_i\geq\pi_{i+2}$ (here we can interpret $\pi_0$ as 
$\infty$ to handle the $i=1$ case). So (1) is proved.

By definition we have 
\[
\xi_1+\xi_2+\cdots+\xi_k=\pi_1+\pi_2+\cdots+\pi_{k-1}+
\max\left\{\pi_k,\frac{\pi_k+\pi_{k+1}}{2}\right\},
\]
so $\xi\geq\pi$, from which the ``if'' direction of (2) follows. 
To prove the
converse, we assume that $\lambda\in\cP_{2n}$ satisfies $\lambda\geq\pi$. Then for any
positive integer $k$, we have
\[
\begin{split}
\lambda_1+\cdots+\lambda_{k-1}&\geq\pi_1+\cdots+\pi_{k-1}\text{ and}\\
\lambda_1+\cdots+\lambda_{k-1}+\lambda_k+\lambda_{k+1}
&\geq\pi_1+\cdots+\pi_{k-1}+\pi_k+\pi_{k+1}.
\end{split}
\]
Taking the average on both sides of these two inequalities and using the fact that
$\lambda_k\geq\lambda_{k+1}$, we deduce that
\[
\lambda_1+\cdots+\lambda_{k-1}+\lambda_k\geq\pi_1+\cdots+\pi_{k-1}
+\frac{\pi_k+\pi_{k+1}}{2}.
\]
Hence
\[
\lambda_1+\cdots+\lambda_k\geq\pi_1+\cdots+\pi_{k-1}+
\max\left\{\pi_k,\frac{\pi_k+\pi_{k+1}}{2}\right\}=\xi_1+\cdots+\xi_k,
\]
as required to prove (2).

The proof of (3) is now easy: if $(\rho;\sigma)\leq(\mu;\nu)$, then
\[ 
(2\rho_1,2\sigma_1,2\rho_2,2\sigma_2,\cdots)\leq
(2\mu_1,2\nu_1,2\mu_2,2\nu_2,\cdots)\leq\Phi^C(\mu;\nu), 
\] 
and (2) applied with
$\lambda=\Phi^C(\mu;\nu)$ implies that
$\Phi^C(\rho;\sigma)\leq \Phi^C(\mu;\nu)$.
\end{proof}

\begin{exam}
When $n=2$, the map $\Phi^C:\cQ_2\to\cP_4^C$ is as follows:
\[
\begin{split}
\Phi^C((2);\varnothing)&=(4),\
\Phi^C((1);(1))=(2^2),\
\Phi^C(\varnothing;(2))=(2^2),\\
\Phi^C((1^2);\varnothing)&=(21^2),\
\Phi^C(\varnothing;(1^2))=(1^4). 
\end{split}
\]
Note in particular that $\Phi^C$ is not injective.
\end{exam}

Recall the definition of the sub-poset of $\cQ_n$ consisting
of \emph{$C$-distinguished} bipartitions:
\[ 
\cQ_n^C=\{(\mu;\nu)\in\cQ_n\,|\,\mu_i\geq\nu_i-1,\,\nu_i\geq\mu_{i+1}-1,
\text{ for all }i\}.
\]
We define a map $\hPhi^C:\cP_{2n}^C\to\cQ_n^C$ as follows. (It is actually a special
case of the general $2$-quotient construction described in 
\cite[Example I.1.8]{macdonald}.)
Start with a partition $\lambda\in\cP_{2n}^C$ and modify it by
halving all even parts, and
replacing any string of odd parts $2k+1,2k+1,\cdots,2k+1,2k+1$ with $k,k+1,\cdots,k,k+1$
(the part $2k+1$ has even multiplicity by assumption). The result is a quasi-partition
of $n$, and we let $\hPhi^C(\lambda)$ be the corresponding bipartition.

\begin{prop} \label{prop:hphiC1}
Let $\pi,\lambda\in\cP_{2n}^C$.
\begin{enumerate}
\item $\hPhi^C(\lambda)\in\cQ_n^C$.
\item If $\pi\leq\lambda$, then $\hPhi^C(\pi)\leq\hPhi^C(\lambda)$.
\end{enumerate} 
\end{prop}
\begin{proof}
Both parts are easy from the definitions.
\end{proof}

\begin{prop} \label{prop:hphiC}
The relationship between $\Phi^C$ and $\hPhi^C$ is as follows.
\begin{enumerate}
\item For any $\lambda\in\cP_{2n}^C$, $\Phi^C(\hPhi^C(\lambda))=\lambda$.
\item For any $(\rho;\sigma)\in\cQ_n$, $(\rho;\sigma)\leq\hPhi^C(\Phi^C(\rho;\sigma))$,
with equality if and only if $(\rho;\sigma)\in\cQ_n^C$.
\item The posets $\cQ_n^C$ and $\cP_{2n}^C$ are isomorphic via the inverse maps
$\Phi^C$ \textup{(}restricted to $\cQ_n^C$\textup{)} and $\hPhi^C$.
\end{enumerate}
\end{prop}
\begin{proof}
Parts (1) and (2) are proved by short calculations, and part (3) follows from these
together with Propositions \ref{prop:phiC} and \ref{prop:hphiC1}.
\end{proof}

\begin{defn}  \label{defn:supC}
For any $(\rho;\sigma)\in\cQ_n$, let $(\rho;\sigma)^C$ denote
$\hPhi^C(\Phi^C(\rho;\sigma))\in\cQ_n^C$. Explicitly, $(\rho;\sigma)^C$ is the
bipartition constructed from $(\rho;\sigma)$ as follows.
Whenever $\rho_i<\sigma_i-1$, 
replace $\rho_i,\sigma_i$ with $\lfloor\frac{\rho_i+\sigma_i}{2}\rfloor,
\lceil\frac{\rho_i+\sigma_i}{2}\rceil$; whenever $\sigma_i<\rho_{i+1}-1$,
replace $\sigma_i,\rho_{i+1}$ with $\lfloor\frac{\sigma_i+\rho_{i+1}}{2}\rfloor,
\lceil\frac{\sigma_i+\rho_{i+1}}{2}\rceil$; as before, it is clear that 
we never have to make overlapping replacements.
\end{defn}

\begin{cor} \label{cor:hphiC}
For $(\rho;\sigma)\in\cQ_n$ and $(\mu;\nu)\in\cQ_n^C$, the following are equivalent.
\begin{enumerate}
\item $(\rho;\sigma)^C=(\mu;\nu)$.
\item $\Phi^C(\rho;\sigma)=\Phi^C(\mu;\nu)$.
\item $(\rho;\sigma)\leq(\mu;\nu)$, and for any $(\tau;\upsilon)\in\cQ_n^C$ such that
$(\rho;\sigma)\leq(\tau;\upsilon)$, we have $(\mu;\nu)\leq(\tau;\upsilon)$.
\end{enumerate}
\end{cor}
\begin{proof}
This follows formally from Proposition \ref{prop:hphiC}.
\end{proof}

With this notation we can reformulate the standard parametrization of nilpotent symplectic
orbits in characteristic $\neq 2$ and the (Zariski) closure ordering on them. 
Let $V$ be a $2n$-dimensional vector space over $\F$ with
a nondegenerate alternating form $\langle\cdot,\cdot\rangle$, let $Sp(V)$ be the
stabilizer of this form, and let $\cN(\fsp(V))$ denote the variety of nilpotent elements
in its Lie algebra.
If $\F$ has positive characteristic $p$ (an odd prime, because 
characteristic $2$ is not allowed in this section), we assume that 
we have a fixed $\F_p$-structure of $V$ such that 
$\langle\cdot,\cdot\rangle$ is defined over $\F_p$; then $Sp(V)$ and $\cN(\fsp(V))$ are
defined over the finite subfield $\F_q$ with $q$ elements for any $q=p^s$.

\begin{thm}[Gerstenhaber, Hesselink] \cite[Theorems 5.1.3, 6.2.5]{cm}
\label{thm:gerst-C}
\begin{enumerate}
\item
The $Sp(V)$-orbits in $\cN(\fsp(V))$ are in bijection with $\cQ_n^C$.
For $(\mu;\nu)\in\cQ_n^C$, the corresponding orbit $\cO_{\mu;\nu}^{C}$ 
has dimension $2(n^2-b(\mu;\nu))$, and 
consists of those
$x\in\cN(\fsp(V))$ whose Jordan type is $\Phi^C(\mu;\nu)$.
\item For $(\rho;\sigma),(\mu;\nu)\in\cQ_n^C$,
\[
\cO_{\rho;\sigma}^{C}\subseteq\overline{\cO_{\mu;\nu}^{C}}\Longleftrightarrow
(\rho;\sigma)\leq(\mu;\nu).
\]
\end{enumerate}
\end{thm}

The statement in \cite{cm} uses, instead of $\cQ_n^C$, the isomorphic poset $\cP_{2n}^C$,
so that the parametrization of orbits is by Jordan types.
As mentioned above, the motivation for switching to the poset $\cQ_n^C$
is that the orbit $\cO_{\mu;\nu}^{C}$, with the trivial local system, corresponds
to the irreducible representation of $W(C_n)$ with character $\chi^{\mu;\nu}$
under the Springer correspondence: this is shown in \cite[Theorem 3.3]{shoji:springer}.

\begin{rmk} \label{rmk:shoji}
Shoji's conventions in \cite{shoji:springer} differ from those above in two ways.
Firstly, Shoji uses Springer's original construction
of Springer representations, from which the one now in use is obtained
by tensoring with the sign representation. Secondly, 
Shoji's parametrization of irreducible characters of $\{\pm 1\}\wr S_n$
differs from ours by transposing both partitions of the bipartition.
\end{rmk}

\begin{rmk}
Theorem \ref{thm:gerst-C} refers only to $C$-distinguished bipartitions, whereas
we defined $\Phi^C$ on all bipartitions.
An irreducible representation of $W(C_n)$ indexed by a bipartition
$(\mu;\nu)$ which is not $C$-distinguished corresponds under the Springer correspondence
to a non-trivial local system on some orbit in $\cN(\fsp(V))$; the Jordan type
of this orbit is an element of $\cP_{2n}^C$, which is in general not equal to
$\Phi^C(\mu;\nu)$. (See \cite[Theorem 3.3]{shoji:springer} or \cite[13.3]{carter}
for the rule.) However, there is a Springer-correspondence interpretation of the
whole map $\Phi^C$ in characteristic $2$, as explained in the next section.
\end{rmk}

\begin{rmk}
The dimension formula for $\cO_{\mu;\nu}^{C}$ stated above is easily deduced from the
formula given in terms of $\Phi^C(\mu;\nu)$ in \cite{cm}, or from the fact that the
dimension of the Springer fibre for a nilpotent element of $\cO_{\mu;\nu}^{C}$ equals
$b(\mu;\nu)$.
\end{rmk}

It is clear from Theorem \ref{thm:gerst-C} that if $\F$ has characteristic $p$,
each orbit $\cO_{\mu;\nu}^C$ is defined over $\F_q$ for any $q=p^s$.
Note that in general,
$\cO_{\mu;\nu}^C(\F_q)$ is the union of multiple $Sp(V)(\F_q)$-orbits; we are not
concerned with the latter kind of orbits.

\begin{prop} \label{prop:poly-C}
For each $(\mu;\nu)\in\cQ_n^C$, there is a polynomial 
$P_{\mu;\nu}^C(t)\in\Z[t]$, independent of $\F$, such that for any finite subfield
$\F_q$ of $\F$, $|\cO_{\mu;\nu}^C(\F_q)|=P_{\mu;\nu}^C(q)$. 
\end{prop}
\begin{proof}
The most convenient reference is the last sentence of \cite[A.6]{lusztig:unip}, which
includes a statement about characteristic $2$ also. 
For historical accuracy, we should note that the result in odd
characteristic has been known for much longer: for example, it follows from
\cite[Section 4]{shoji:green} that $|\cO_{\mu;\nu}^C(\F_q)|$ is a rational function 
of $q$ independent of the characteristic, and by the Grothendieck Trace Formula, this
rational function must be an integer polynomial.
\end{proof}

In \cite{lusztig:notes}, the polynomial $P_{\mu;\nu}^C(t)$ is given an alternative interpretation in terms of
equivariant Betti numbers of nilpotent orbits, in the case where $\F=\C$:
\begin{equation} \label{eqn:betti-C}
P_{\mu;\nu}^C(t)
=t^{2(n^2-b(\mu;\nu))}\,\prod_{i=1}^n(1-t^{-2i})\,
\sum_{j\geq 0}\dim H_{2j}^{Sp(V)}(\cO_{\mu;\nu}^C)\,
t^{-j}.
\end{equation}
An explicit rational-function formula for the power series on the 
right-hand side is given in \cite[Lemmas 5.2, 5.3]{lusztig:notes}; it follows from 
this that 
\begin{equation} \label{eqn:explicit-C}
P_{\mu;\nu}^C(t)
=\frac{t^{2(n^2-b(\mu;\nu))}\,\displaystyle\prod_{i=1}^n(1-t^{-2i})}
{\displaystyle\prod_{a\geq 1}\prod_{i=1}^{f_a^C(\mu;\nu)}(1-t^{-2i})},
\end{equation}
where $f_a^C(\mu;\nu)=\lfloor\frac{m_a(\Phi^C(\mu;\nu))}{2}\rfloor$, and
$m_a(\Phi^C(\mu;\nu))$ means the multiplicity of $a$ as a part of the partition
$\Phi^C(\mu;\nu)$.
Hence $P_{\mu;\nu}^C(t)$ is actually 
a polynomial in $t^2$. We will not need to use the explicit formula 
\eqref{eqn:explicit-C} in this paper.

\subsection{The nilpotent cone of type $B_n$}

The situation in type $B_n$ is very similar. Let 
$\cP_{2n+1}^B$ denote the sub-poset of $\cP_{2n+1}$
consisting of partitions of $2n+1$ in which every even part has
even multiplicity.
We define a map $\Phi^B:\cQ_n\to\cP_{2n+1}^B$
as follows. For any $(\rho;\sigma)\in\cQ_n$, consider the infinite sequence
\begin{equation*}
(2\rho_1+1,2\sigma_1-1,2\rho_2+1,2\sigma_2-1,2\rho_3+1,2\sigma_3-1,\cdots).
\end{equation*}
Whenever two consecutive terms $s,t$ satisfy $s<t$, replace them with
$\frac{s+t}{2},\frac{s+t}{2}$ (note that $\frac{s+t}{2}\in\N$); 
clearly we never have to make overlapping replacements. 
The terms in the infinite tail $-1,1,-1,1,\cdots$ are all replaced with $0$,
so the result is a composition $\Phi^B(\rho;\sigma)$ of $2n+1$.

\begin{prop} \label{prop:phiB}
Let $(\rho;\sigma),(\mu;\nu)\in\cQ_n$.
\begin{enumerate}
\item
$\Phi^B(\rho;\sigma)\in\cP_{2n+1}^B$.
\item
If $\lambda\in\cP_{2n+1}$, then $\lambda\geq
(2\rho_1+1,2\sigma_1-1,2\rho_2+1,2\sigma_2-1,\cdots)$ if and only if
$\lambda\geq\Phi^B(\rho;\sigma)$.
\item
If $(\rho;\sigma)\leq(\mu;\nu)$, then $\Phi^B(\rho;\sigma)\leq\Phi^B(\mu;\nu)$. 
\end{enumerate}
\end{prop}
\begin{proof}
Essentially the same as the proof of Proposition \ref{prop:phiC}.
\end{proof}

\begin{exam}
When $n=2$, the map $\Phi^B:\cQ_2\to\cP_5^B$ is as follows:
\[
\begin{split}
\Phi^B((2);\varnothing)&=(5),\
\Phi^B((1);(1))=(31^2),\
\Phi^B(\varnothing;(2))=(2^21),\\
\Phi^B((1^2);\varnothing)&=(31^2),\
\Phi^B(\varnothing;(1^2))=(1^5). 
\end{split}
\]
\end{exam}

Recall the definition of the sub-poset of $\cQ_n$ consisting
of \emph{$B$-distinguished} bipartitions:
\[ 
\cQ_n^B=\{(\mu;\nu)\in\cQ_n\,|\,\mu_i\geq\nu_i-2,\,\nu_i\geq\mu_{i+1},
\text{ for all }i\}.
\]
We define a map $\hPhi^B:\cP_{2n+1}^B\to\cQ_n^B$ as follows.
Start with $\lambda\in\cP_{2n+1}^B$.
If $\lambda_i$ is odd, replace it with $\frac{\lambda_i+(-1)^i}{2}$;
and for every string of equal even parts $\lambda_i=\lambda_{i+1}=\cdots=2k$ 
with $\lambda_{i-1}>2k$, replace them
with $k,k,\cdots$ if $i$ is even or with $k-1,k+1,k-1,k+1,\cdots$ if $i$ is odd. 
(This applies to the infinite tail of zeroes, when $k=0$ and $i$ is necessarily
even; they all remain zero.) The result is a quasi-partition of $n$,
and we let $\hPhi^B(\lambda)$ be the corresponding bipartition. 

\begin{prop} \label{prop:hphiB1}
Let $\pi,\lambda\in\cP_{2n+1}^B$.
\begin{enumerate}
\item $\hPhi^B(\lambda)\in\cQ_n^B$.
\item If $\pi\leq\lambda$, then $\hPhi^B(\pi)\leq\hPhi^B(\lambda)$.
\end{enumerate} 
\end{prop}
\begin{proof}
Both parts are easy from the definitions.
\end{proof}

\begin{prop} \label{prop:hphiB}
The relationship between $\Phi^B$ and $\hPhi^B$ is as follows.
\begin{enumerate}
\item For any $\lambda\in\cP_{2n+1}^B$, $\Phi^B(\hPhi^B(\lambda))=\lambda$.
\item For any $(\rho;\sigma)\in\cQ_n$, $(\rho;\sigma)\leq\hPhi^B(\Phi^B(\rho;\sigma))$,
with equality if and only if $(\rho;\sigma)\in\cQ_n^B$.
\item The posets $\cQ_n^B$ and $\cP_{2n+1}^B$ are isomorphic via the inverse maps
$\Phi^B$ \textup{(}restricted to $\cQ_n^B$\textup{)} and $\hPhi^B$.
\end{enumerate}
\end{prop}
\begin{proof}
Again, parts (1) and (2) are proved by short calculations, and part (3) follows from these
together with Propositions \ref{prop:phiB} and \ref{prop:hphiB1}.
\end{proof}

\begin{defn} \label{defn:supB}
For any $(\rho;\sigma)\in\cQ_n$, let $(\rho;\sigma)^B$ denote 
$\hPhi^B(\Phi^B(\rho;\sigma))\in\cQ_n^B$. Explicitly, $(\rho;\sigma)^B$ is the
bipartition constructed from $(\rho;\sigma)$ as follows.
Whenever $\rho_i<\sigma_i-2$, 
replace $\rho_i,\sigma_i$ with $\lceil\frac{\rho_i+\sigma_i}{2}\rceil-1,
\lfloor\frac{\rho_i+\sigma_i}{2}\rfloor+1$; whenever $\sigma_i<\rho_{i+1}$,
replace $\sigma_i,\rho_{i+1}$ with $\lceil\frac{\sigma_i+\rho_{i+1}}{2}\rceil,
\lfloor\frac{\sigma_i+\rho_{i+1}}{2}\rfloor$; as before, it is clear that 
we never have to make overlapping replacements.
\end{defn}

\begin{cor} \label{cor:hphiB}
For $(\rho;\sigma)\in\cQ_n$ and $(\mu;\nu)\in\cQ_n^B$, the following are equivalent.
\begin{enumerate}
\item $(\rho;\sigma)^B=(\mu;\nu)$.
\item $\Phi^B(\rho;\sigma)=\Phi^B(\mu;\nu)$.
\item $(\rho;\sigma)\leq(\mu;\nu)$, and for any $(\tau;\upsilon)\in\cQ_n^B$ such that
$(\rho;\sigma)\leq(\tau;\upsilon)$, we have $(\mu;\nu)\leq(\tau;\upsilon)$.
\end{enumerate}
\end{cor}
\begin{proof}
This follows formally from Proposition \ref{prop:hphiB}.
\end{proof}

Let $\tV$ be a $(2n+1)$-dimensional vector space over $\F$ with
a nondegenerate quadratic form $\tQ$, let $SO(\tV)$ be the identity component of the
stabilizer of this form, and let $\cN(\fo(\tV))$ denote the variety of nilpotent elements
in its Lie algebra. If $\F$ has characteristic $p$ (an odd prime), we assume that
we have a fixed $\F_p$-structure on $\tV$ such that $\tQ$ is defined over $\F_p$; then
$SO(\tV)$ and $\cN(\fo(\tV))$ are defined over $\F_q$ for any $q=p^s$.

\begin{thm}[Gerstenhaber, Hesselink] \cite[Theorems 5.1.2, 6.2.5]{cm}
\label{thm:gerst-B}
\begin{enumerate}
\item
The $SO(\tV)$-orbits in $\cN(\fo(\tV))$ are in bijection with $\cQ_n^B$.
For $(\mu;\nu)\in\cQ_n^B$, the corresponding orbit $\cO_{\mu;\nu}^{B}$ 
has dimension $2(n^2-b(\mu;\nu))$, and 
consists of those
$x\in\cN(\fo(\tV))$ whose Jordan type is $\Phi^B(\mu;\nu)$.
\item For $(\rho;\sigma),(\mu;\nu)\in\cQ_n^B$,
\[
\cO_{\rho;\sigma}^{B}\subseteq\overline{\cO_{\mu;\nu}^{B}}\Longleftrightarrow
(\rho;\sigma)\leq(\mu;\nu).
\]
\end{enumerate}
\end{thm}

As in type $C$, this result is usually stated in terms of
$\cP_{2n+1}^B$ instead of $\cQ_n^B$. Again, it follows from
\cite[Theorem 3.3]{shoji:springer}
that the orbit $\cO_{\mu;\nu}^{B}$, with the trivial local system, corresponds
to the irreducible representation of $W(B_n)$ with character $\chi^{\mu;\nu}$
under the Springer correspondence (bearing in mind Remark \ref{rmk:shoji}).


It is clear from Theorem \ref{thm:gerst-B} that, if $\F$ has characteristic $p$, 
each orbit $\cO_{\mu;\nu}^B$ is defined over $\F_q$ for any $q=p^s$. As in type $C$,
we are not concerned with the $O(\tV)(\F_q)$-orbits into which $\cO_{\mu;\nu}^B(\F_q)$
splits.

\begin{prop} \label{prop:poly-B}
For each $(\mu;\nu)\in\cQ_n^B$, there is a polynomial 
$P_{\mu;\nu}^B(t)\in\Z[t]$, independent of $\F$, such that for any finite subfield
$\F_q$ of $\F$, $|\cO_{\mu;\nu}^B(\F_q)|=P_{\mu;\nu}^B(q)$. 
\end{prop}
\begin{proof}
The same comments as in the proof of Proposition \ref{prop:poly-C} apply here.
\end{proof}

As in type $C$, \cite{lusztig:notes} gives, in the $\F=\C$ case,
\begin{equation} \label{eqn:betti-B}
P_{\mu;\nu}^B(t)
=t^{2(n^2-b(\mu;\nu))}\,\prod_{i=1}^n(1-t^{-2i})\,
\sum_{j\geq 0}\dim H_{2j}^{SO(\tV)}(\cO_{\mu;\nu}^B)\,
t^{-j},
\end{equation}
and also an explicit formula for
$P_{\mu;\nu}^B(t)$ analogous to \eqref{eqn:explicit-C}.

\subsection{The exotic nilpotent cone}

Hitherto we have used proper sub-posets of $\cQ_n$ to parametrize orbits; 
we will now see that the full poset $\cQ_n$ 
occurs in the parametrization of orbits in the exotic
nilpotent cone. 

As in \S2.1, $V$ denotes a $2n$-dimensional vector space over $\F$ with
a nondegenerate alternating form $\langle\cdot,\cdot\rangle$. Define
\[ S=\{x\in\End(V)\,|\,\langle xv,v\rangle =0,\text{ for all }v\in V\}. \] 
Under our assumption
that $\F$ has characteristic different from $2$, 
$S$ is the $Sp(V)$-invariant complementary subspace to
$\fsp(V)$ in $\fgl(V)=\End(V)$.
Syu Kato's \emph{exotic nilpotent cone} of type $C_n$ is 
\[ \fN=V\times(S\cap\cN(\fgl(V)))=\{(v,x)\in V\times S\,|\,x\text{ nilpotent}\}. \] 

To describe the parametrization of $Sp(V)$-orbits in $\fN$, we first need to recall
the parametrization of $GL(V)$-orbits in the 
\emph{enhanced nilpotent cone} $V\times\cN(\fgl(V))$. As explained in \cite[Section 2]{ah},
these are in bijection with bipartitions of $2n$. For $(\tau;\varphi)\in\cQ_{2n}$,
the orbit $\cO_{\tau;\varphi}\subset V\times\cN(\fgl(V))$ consists of
those $(v,x)$ such that
\begin{itemize}
\item $x$ has Jordan type $\tau+\varphi$
(the partition obtained by adding $(\tau_1,\tau_2,\cdots)$ and 
$(\varphi_1,\varphi_2,\cdots)$ term-by-term), and 
\item $v$ belongs to the open
$GL(V)^x$-orbit in $\sum_i x^{\varphi_i}(\ker(x^{\tau_i+\varphi_i}))$, where
$GL(V)^x$ is the stabilizer of $x$ in $GL(V)$. 
\end{itemize}
For a more explicit
description, see \cite[Proposition 2.3]{ah}.

For any partition $\lambda$, let $\lambda\cup\lambda$ denote
the `duplicated' partition $(\lambda_1,\lambda_1,\lambda_2,\lambda_2,\cdots)$.

\begin{thm}[Kato, Achar--Henderson] \textup{\cite{kato:exotic} for (1), 
\cite[Section 6]{ah} for (2).}
\label{thm:exotic}
\begin{enumerate}
\item 
The $Sp(V)$-orbits in $\fN$ are in bijection with $\cQ_n$. For 
$(\mu;\nu)\in\cQ_n$, the corresponding orbit $\Oo_{\mu;\nu}$ 
has dimension $2(n^2-b(\mu;\nu))$ and 
is the intersection of $\fN$ with the orbit
$\cO_{\mu\cup\mu;\nu\cup\nu}$
in the enhanced nilpotent cone $V\times\cN(\fgl(V))$. 
\item 
For $(\rho;\sigma),(\mu;\nu)\in\cQ_n$,
\[
\Oo_{\rho;\sigma}\subseteq\overline{\Oo_{\mu;\nu}}\Longleftrightarrow
(\rho;\sigma)\leq(\mu;\nu).
\]
\end{enumerate}
\end{thm}

From the above description of orbits in the enhanced nilpotent cone, it follows that
$\Oo_{\mu;\nu}$ consists of those $(v,x)\in\fN$ such that
\begin{itemize}
\item $x$ has Jordan type $(\mu+\nu)\cup(\mu+\nu)$, and 
\item $v$ belongs to the open $Sp(V)^x$-orbit in
$\sum_i x^{\nu_i}(\ker(x^{\mu_i+\nu_i}))$. 
\end{itemize}
For an explicit representative of each
orbit $\Oo_{\mu;\nu}$, see the proof of \cite[Theorem 6.1]{ah}.
It is clear that if $\F$ has characteristic $p$, each orbit $\Oo_{\mu;\nu}$ is defined
over $\F_q$ for any $q=p^s$.

\begin{prop} \label{prop:poly-ex}
For each $(\mu;\nu)\in\cQ_n$, there is a polynomial 
$P_{\mu;\nu}(t)\in\Z[t]$, independent of $\F$, such that for any finite subfield
$\F_q$ of $\F$, $|\Oo_{\mu;\nu}(\F_q)|=P_{\mu;\nu}(q)$. 
\end{prop}

\begin{proof}
This is easier than Propositions \ref{prop:poly-C} and \ref{prop:poly-B}, as
the stabilizer $Sp(V)^{(v,x)}$ in $Sp(V)$ of a point $(v,x)$ in $\Oo_{\mu;\nu}(\F_q)$ 
is connected (see 
\cite[Proposition 4.5]{kato:exotic}), so $\frac{|Sp(V)(\F_q)|}{|\Oo_{\mu;\nu}(\F_q)|}$
equals $|Sp(V)^{(v,x)}(\F_q)|$. By the
Grothendieck Trace Formula, showing that
$|\Oo_{\mu;\nu}(\F_q)|$ is a rational function in $q$ is enough to imply that it is
an integer polynomial in $q$, so it suffices to show that $|Sp(V)^{(v,x)}(\F_q)|$
is given by an integer polynomial in $q$ independent of $\F$.
This follows from the fact that the reductive
quotient of $Sp(V)^{(v,x)}$ is a product of symplectic groups, 
whose ranks are determined inductively
in \cite[Section 4]{springer:exotic}, independently of $\F$. 
\end{proof}

By the same argument, it follows that if $\F=\C$,
\begin{equation} \label{eqn:betti-ex}
P_{\mu;\nu}(t)
=t^{2(n^2-b(\mu;\nu))}\,\prod_{i=1}^n(1-t^{-2i})\,
\sum_{j\geq 0}\dim H_{2j}^{Sp(V)}(\Oo_{\mu;\nu})\,
t^{-j}.
\end{equation}
The following explicit formula for
$P_{\mu;\nu}(t)$ is proved in \cite[Corollary 3.16]{sun}:
\begin{equation}
P_{\mu;\nu}(t)=
\frac{t^{2(n^2-b(\mu;\nu))}\displaystyle\prod_{i=1}^n (1-t^{-2i})}
{\displaystyle\prod_{a\in J(\mu;\nu)}\left(\prod_{i=1}^{m_a(\mu+\nu)-1}(1-t^{-2i})\right)
\prod_{a\notin J(\mu;\nu)}\left(\prod_{i=1}^{m_a(\mu+\nu)}(1-t^{-2i})\right)},
\end{equation}
where $m_a(\mu+\nu)$ means the multiplicity of the positive integer
$a$ as a part of $\mu+\nu$,
and $J(\mu;\nu)$ is the set of $a$ such that
for some $i\geq 1$, $\mu_i+\nu_i=a$ and $\mu_{i+1}<\mu_i$, and for some
$j\geq 1$ (possibly the same as $i$), $\mu_j+\nu_j=a$ and $\nu_{j-1}>\nu_j$
(here we interpret $\nu_0$ as $\infty$ if $j=1$). We will use this formula only in
Example \ref{exam:minspecial}.


\subsection{Pieces of the exotic nilpotent cone}

Now we can define the pieces of $\fN$ mentioned in the introduction.
For explicit examples, see Examples \ref{exam:minspecial} and 
\ref{exam:maybenotsmoothspecial}.

\begin{defn} \label{defn:pieces}
For any $(\mu;\nu)\in\cQ_n^B$, we define the corresponding \emph{type-$B$ piece} of
$\fN$ to be the following disjoint union of orbits:
\[
\Tt_{\mu;\nu}^B=\bigdisjunion_{\substack{(\rho;\sigma)\in\cQ_n\\
(\rho;\sigma)^B=(\mu;\nu)}}\Oo_{\rho;\sigma}.
\]
For any $(\mu;\nu)\in\cQ_n^C$, we define the corresponding \emph{type-$C$ piece} of
$\fN$ to be
\[
\Tt_{\mu;\nu}^C=\bigdisjunion_{\substack{(\rho;\sigma)\in\cQ_n\\
(\rho;\sigma)^C=(\mu;\nu)}}\Oo_{\rho;\sigma}.
\]
\end{defn}

Combining Theorem \ref{thm:gerst-C}(2) and Theorem \ref{thm:gerst-B}(2) with
Corollaries \ref{cor:hphiC} and \ref{cor:hphiB}, we see that
\begin{equation}
\begin{split}
\Tt_{\mu;\nu}^B&=\;\overline{\Oo_{\mu;\nu}}\,\setminus
\bigcup_{\substack{(\tau;\upsilon)\in\cQ_n^B\\(\tau;\upsilon)<(\mu;\nu)}}
\overline{\Oo_{\tau;\upsilon}}\quad\text{and}
\\
\quad
\Tt_{\mu;\nu}^C&=\;\overline{\Oo_{\mu;\nu}}\,\setminus
\bigcup_{\substack{(\tau;\upsilon)\in\cQ_n^C\\(\tau;\upsilon)<(\mu;\nu)}}
\overline{\Oo_{\tau;\upsilon}},
\end{split}
\end{equation}
so $\Tt_{\mu;\nu}^B$ and $\Tt_{\mu;\nu}^C$ are locally closed subvarieties
of $\fN$.

We can now state the main result of this paper.

\begin{thm} \label{thm:main}
Suppose that $\F$ has odd characteristic $p$.
\begin{enumerate}
\item For any $(\mu;\nu)\in\cQ_n^B$ and any $q=p^s$, 
$|\Tt_{\mu;\nu}^B(\F_q)|=|\cO_{\mu;\nu}^B(\F_q)|$.
Equivalently, we have an equality of polynomials:
\[ \sum_{\substack{(\rho;\sigma)\in\cQ_n\\
(\rho;\sigma)^B=(\mu;\nu)}} P_{\rho;\sigma}(t)=P_{\mu;\nu}^B(t). \]
\item For any $(\mu;\nu)\in\cQ_n^C$ and any $q=p^s$, 
$|\Tt_{\mu;\nu}^C(\F_q)|=|\cO_{\mu;\nu}^C(\F_q)|$.
Equivalently, we have an equality of polynomials:
\[ \sum_{\substack{(\rho;\sigma)\in\cQ_n\\
(\rho;\sigma)^C=(\mu;\nu)}} P_{\rho;\sigma}(t)=P_{\mu;\nu}^C(t). \]
\end{enumerate}
\end{thm}

In view of \eqref{eqn:betti-C}, \eqref{eqn:betti-B}, \eqref{eqn:betti-ex}, and
\cite[Lemma 5.2]{lusztig:notes}, one can also intepret Theorem \ref{thm:main}(1) as saying
that, when $\F=\C$, 
the varieties $\Tt_{\mu;\nu}^B$ and $\cO_{\mu;\nu}^B$ have the
same equivariant Betti numbers (for their respective groups $Sp(V)$ and $SO(\tV)$),
and Theorem \ref{thm:main}(2) as saying
that $\Tt_{\mu;\nu}^C$ and $\cO_{\mu;\nu}^C$ have the
same equivariant Betti numbers (for the group $Sp(V)$).

Our proof of Theorem \ref{thm:main} does not directly relate the varieties
$\Tt_{\mu;\nu}^B$ and $\cO_{\mu;\nu}^B$, or $\Tt_{\mu;\nu}^C$ and $\cO_{\mu;\nu}^C$,
in good characteristic. Instead, we prove each equality of polynomials in the case
when $t$ is a power of $2$ (which obviously suffices). This
requires switching to the case of characteristic $2$ and finding explicit
connections there between the three nilpotent cones $\cN(\fsp(V))$, 
$\cN(\fo(\tV))$, and $\fN$. 
Part (2) of Theorem \ref{thm:main} will be proved (or rather, deduced formally
from results of Spaltenstein, Kato, Lusztig, and Xue) in Section 3, 
and part (1) will be proved in Section 4.

In this approach, we never need to use
the explicit formulas for $P_{\mu;\nu}^C(t)$, $P_{\mu;\nu}^B(t)$, and $P_{\mu;\nu}(t)$ 
referred to above.
It is probably possible to give a combinatorial proof of
Theorem \ref{thm:main} using these formulas, but we prefer our method because it
gives some insight into the relationship between the varieties.

To conclude this section, we deduce Lusztig's result on special
pieces from Theorem \ref{thm:main}. Recall that the poset of \emph{special} 
bipartitions is defined by
\[
\cQ_n^\circ=\cQ_n^B\cap\cQ_n^C=\{(\mu;\nu)\in\cQ_n\,|\,\mu_i\geq\nu_i-1,
\,\nu_i\geq\mu_{i+1},\text{ for all }i\}.
\]
The following result was observed by Spaltenstein in the cases corresponding to
orbits in $\cN(\fsp(V))$ or $\cN(\fo(\tV))$ (i.e.\ where $(\rho;\sigma)\in\cQ_n^C$
or $(\rho;\sigma)\in\cQ_n^B$). 

\begin{prop} \label{prop:special}
For any $(\rho;\sigma)\in\cQ_n$, there is a unique
$(\rho;\sigma)^\circ\in\cQ_n^\circ$ satisfying:
\begin{enumerate}
\item $(\rho;\sigma)\leq(\rho;\sigma)^\circ$, and
\item $(\rho;\sigma)^\circ\leq(\tau;\upsilon)$
for any $(\tau;\upsilon)\in\cQ_n^\circ$ such that 
$(\rho;\sigma)\leq(\tau;\upsilon)$.
\end{enumerate}
\end{prop}

\begin{proof}
Define $(\rho;\sigma)^\circ$ to be $((\rho;\sigma)^C)^B$.
From Definitions \ref{defn:supC} and \ref{defn:supB}, we see that 
$(\rho;\sigma)^\circ$ can be constructed from $(\rho;\sigma)$ as
follows. Whenever $\rho_i<\sigma_i-1$, replace $\rho_i,\sigma_i$ with
$\lfloor\frac{\rho_i+\sigma_i}{2}\rfloor,\lceil\frac{\rho_i+\sigma_i}{2}\rceil$; 
whenever $\sigma_i<\rho_{i+1}$, replace $\sigma_i,\rho_{i+1}$ with 
$\lceil\frac{\sigma_i+\rho_{i+1}}{2}\rceil,\lfloor\frac{\sigma_i+\rho_{i+1}}{2}\rfloor$;
as usual, we never have to make overlapping replacements. It is clear that
the bipartition $(\rho;\sigma)^\circ$ constructed in this way is special and satisfies
$(\rho;\sigma)\leq(\rho;\sigma)^\circ$.
Suppose that another special bipartition $(\tau;\upsilon)$ satisfied
$(\rho;\sigma)\leq(\tau;\upsilon)$. Since $(\tau;\upsilon)\in\cQ_n^C$, Corollary
\ref{cor:hphiC} implies that 
$(\rho;\sigma)^C\leq(\tau;\upsilon)$; then since
$(\tau;\upsilon)\in\cQ_n^B$, Corollary \ref{cor:hphiB} implies that 
$(\rho;\sigma)^\circ\leq(\tau;\upsilon)$.
The uniqueness of $(\rho;\sigma)^\circ$ follows formally.
\end{proof}

\begin{rmk}
When $(\rho;\sigma)\in\cQ_n^B$ or $(\rho;\sigma)\in\cQ_n^C$, the irreducible
representation of $\{\pm 1\}\wr S_n$ labelled by $(\rho;\sigma)^\circ$ is in the 
same family, in Lusztig's sense, as that labelled by $(\rho;\sigma)$ 
(see \cite[Theorem 0.2]{lusztig:notes}). This is not the case in general: for example,
$(\varnothing;(3))^\circ=((1);(2))$ is in a different family from $(\varnothing;(3))$.
\end{rmk}

As a consequence of Proposition \ref{prop:special}, $\cN(\fo(\tV))$, $\cN(\fsp(V))$, and $\fN$ each have a partition into
locally closed special pieces.
\begin{defn} \label{defn:special}
For any $(\mu;\nu)\in\cQ_n^\circ$, define the corresponding \emph{special piece}
of $\cN(\fo(\tV))$, $\cN(\fsp(V))$, or $\fN$ respectively to be:
\[
\begin{split}
\cS_{\mu;\nu}^B&=\bigdisjunion_{\substack{(\rho;\sigma)\in\cQ_n^B\\
(\rho;\sigma)^\circ=(\mu;\nu)}}\cO_{\rho;\sigma}^B
\;=\;\overline{\cO_{\mu;\nu}^B}\,\setminus
\bigcup_{\substack{(\tau;\upsilon)\in\cQ_n^\circ\\(\tau;\upsilon)<(\mu;\nu)}}
\overline{\cO_{\tau;\upsilon}^B},\\
\cS_{\mu;\nu}^C&=\bigdisjunion_{\substack{(\rho;\sigma)\in\cQ_n^C\\
(\rho;\sigma)^\circ=(\mu;\nu)}}\cO_{\rho;\sigma}^C
\;=\;\overline{\cO_{\mu;\nu}^C}\,\setminus
\bigcup_{\substack{(\tau;\upsilon)\in\cQ_n^\circ\\(\tau;\upsilon)<(\mu;\nu)}}
\overline{\cO_{\tau;\upsilon}^C},\\
\Ss_{\mu;\nu}&=\bigdisjunion_{\substack{(\rho;\sigma)\in\cQ_n\\
(\rho;\sigma)^\circ=(\mu;\nu)}}\Oo_{\rho;\sigma}
\;=\;\overline{\Oo_{\mu;\nu}}\,\setminus
\bigcup_{\substack{(\tau;\upsilon)\in\cQ_n^\circ\\(\tau;\upsilon)<(\mu;\nu)}}
\overline{\Oo_{\tau;\upsilon}}.
\end{split}
\]
\end{defn}

The equality $|\cS_{\mu;\nu}^B(\F_q)|=|\cS_{\mu;\nu}^C(\F_q)|$ was proved in
\cite[Theorem 0.9, 6.9]{lusztig:notes} by directly computing both sides of the
equivalent polynomial equality.
We now have a new proof which goes via the exotic nilpotent cone.

\begin{thm} \label{thm:special}
Suppose that $\F$ has odd characteristic $p$.
For any $(\mu;\nu)\in\cQ_n^\circ$ and any $q=p^s$,
$|\cS_{\mu;\nu}^B(\F_q)|=|\Ss_{\mu;\nu}(\F_q)|=|\cS_{\mu;\nu}^C(\F_q)|$.
Equivalently, we have an equality of polynomials:
\[
\sum_{\substack{(\rho;\sigma)\in\cQ_n^B\\
(\rho;\sigma)^\circ=(\mu;\nu)}} P_{\rho;\sigma}^B(t)=
\sum_{\substack{(\rho;\sigma)\in\cQ_n\\
(\rho;\sigma)^\circ=(\mu;\nu)}} P_{\rho;\sigma}(t)=
\sum_{\substack{(\rho;\sigma)\in\cQ_n^C\\
(\rho;\sigma)^\circ=(\mu;\nu)}} P_{\rho;\sigma}^C(t).
\]
\end{thm}

\begin{proof}
Since $\cQ_n^\circ\subseteq\cQ_n^B$ and $\cQ_n^\circ\subseteq\cQ_n^C$, 
the special pieces in $\fN$ are disjoint
unions of type-$B$ pieces, and also disjoint unions of type-$C$
pieces (in fact, the partition of $\fN$ into special pieces is the finest partition
for which this holds). Explicitly, for any $(\mu;\nu)\in\cQ_n^\circ$,
\begin{equation}
\Ss_{\mu;\nu}=\bigdisjunion_{\substack{(\rho;\sigma)\in\cQ_n^B\\(\rho;\sigma)^\circ=(\mu;\nu)}}
\Tt_{\rho;\sigma}^B
=\bigdisjunion_{\substack{(\rho;\sigma)\in\cQ_n^C\\(\rho;\sigma)^\circ=(\mu;\nu)}}
\Tt_{\rho;\sigma}^C.
\end{equation}
So the result is an immediate consequence of Theorem \ref{thm:main}.
\end{proof}

\begin{exam} \label{exam:minspecial}
Take $n=3$. We have $((1);(1^2))\in\cQ_3^\circ$, and 
$(\varnothing;(21))^\circ=((1^3);\varnothing)^\circ=((1);(1^2))$; the only other
bipartition less than $((1);(1^2))$ is $(\varnothing;(1^3))$ which is itself special.
So the special piece $\Ss_{(1);(1^2)}$ in the exotic nilpotent cone is the union of
the orbits $\Oo_{(1);(1^2)}$, $\Oo_{\varnothing;(21)}$, and $\Oo_{(1^3);\varnothing}$.
The numbers of 
$\F_q$-points of these orbits are:
\[
\begin{split}
|\Oo_{(1);(1^2)}(\F_q)|&=(q^4-1)(q^6-1),\\
|\Oo_{\varnothing;(21)}(\F_q)|&=(q^2+1)(q^6-1),\\
|\Oo_{(1^3);\varnothing}(\F_q)|&=q^6-1.
\end{split}
\]
Now $(\varnothing;(21))$ is in $\cQ_3^B$ but not $\cQ_3^C$, and
$((1^3);\varnothing)$ is in $\cQ_3^C$ but not $\cQ_3^B$. So the special piece breaks
into type-$B$ and type-$C$ pieces as follows:
\[
\begin{split}
\Ss_{(1);(1^2)}&=\Tt_{(1);(1^2)}^B\disjunion\Tt_{\varnothing;(21)}^B,\text{ where }\\
&\Tt_{(1);(1^2)}^B=\Oo_{(1);(1^2)}\disjunion\Oo_{(1^3);\varnothing},\
\Tt_{\varnothing;(21)}^B=\Oo_{\varnothing;(21)},\\
\Ss_{(1);(1^2)}&=\Tt_{(1);(1^2)}^C\disjunion\Tt_{(1^3);\varnothing}^C,\text{ where }\\
&\Tt_{(1);(1^2)}^C=\Oo_{(1);(1^2)}\disjunion\Oo_{\varnothing;(21)},\
\Tt_{(1^3);\varnothing}^C=\Oo_{(1^3);\varnothing}.
\end{split}
\]
We have
\[
\begin{split}
|\Tt_{(1);(1^2)}^B(\F_q)|&=(q^4-1)(q^6-1)+q^6-1=q^4(q^6-1),\\
|\Tt_{\varnothing;(21)}^B(\F_q)|&=(q^2+1)(q^6-1).
\end{split}
\]
In accordance with Theorem \ref{thm:main}(1), these are also the numbers of $\F_q$-points
in the orbits $\cO_{(1);(1^2)}^B$ 
of Jordan type $(31^4)$ and $\cO_{\varnothing;(21)}^B$
of Jordan type $(2^21^3)$ in $\cN(\fo_7)$. Similarly,

\[
\begin{split}
|\Tt_{(1);(1^2)}^C(\F_q)|&=(q^4-1)(q^6-1)+(q^2+1)(q^6-1)=(q^4+q^2)(q^6-1),\\
|\Tt_{(1^3);\varnothing}^C(\F_q)|&=q^6-1,
\end{split}
\]
which in accordance with Theorem \ref{thm:main}(2) are also the numbers of $\F_q$-points
in the orbits $\cO_{(1);(1^2)}^C$ of Jordan type $(2^21^2)$ and $\cO_{(1^3);\varnothing}^C$
of Jordan type $(21^4)$ in $\cN(\fsp_6)$. The number of $\F_q$-points in the special piece
labelled by $((1);(1^2))$, in any one of the three cones, is
$(q^4+q^2+1)(q^6-1)$.
\end{exam}
\section{Comparing exotic and symplectic nilpotent cones in characteristic $2$}
In this section, $\F$ denotes an algebraically closed field of characteristic $2$,
and $V$ is a $2n$-dimensional vector space over $\F$. For concreteness, we choose
a basis $e_1,e_2,\cdots,e_{2n}$ of $V$ and use it to
identify $\End(V)$ with the ring of $2n\times 2n$ matrices over $\F$.
We fix a nondegenerate quadratic form $Q$ on $V$:
\[ Q(a_1e_1+a_2e_2+\cdots+a_{2n-1}e_{2n-1}+a_{2n}e_{2n})
=a_1a_{2n}+a_2a_{2n-1}+\cdots+a_na_{n+1}. \]
The associated nondegenerate
symmetric form is defined by 
\[ \langle v,w\rangle=Q(v+w)-Q(v)-Q(w). \]
Explicitly,
\begin{equation} 
\langle a_1e_1+\cdots+a_{2n}e_{2n}, b_1e_1+\cdots+b_{2n}e_{2n}\rangle=
a_1b_{2n}+a_2b_{2n-1}+\cdots+a_{2n}b_1.
\end{equation}
Since $\F$ has characteristic $2$, $\langle v,v\rangle=0$ for all $v\in V$, so 
this is actually an alternating form.

Let $O(V)$ and $Sp(V)$ be the stabilizers in $GL(V)$
of the quadratic form $Q$ and the alternating form $\langle\cdot,\cdot\rangle$
respectively. Then $O(V)$ is clearly a closed subgroup of $Sp(V)$, and both are defined
over $\F_{q}$ for any $q=2^s$.
The identity component 
of $O(V)$ is a simple group of type $D_n$ (for $n\geq 4$). The symplectic
group $Sp(V)$ is connected and simple of type $C_n$ (for $n\geq 2$). A reference for these and subsequent facts is
\cite[Exercises 7.4.7]{springer}.

The Lie algebra of $Sp(V)$ is, as in all characteristics, 
\[ \fsp(V)=\{x\in\End(V)\,|\,\langle xv,w\rangle=\langle xw,v\rangle,\text{ for all }
v,w\in V\}. \]
As a representation of $Sp(V)$, it is isomorphic to $S^2(V^*)$, the space of
symmetric bilinear forms on $V$, via the isomorphism which sends $x\in\fsp(V)$ 
to the form
$(v,w)\mapsto\langle xv,w\rangle$. A new feature of characteristic $2$ is that
$\fsp(V)$ is not an irreducible representation of $Sp(V)$,
because it contains as a subrepresentation (and Lie ideal) 
the Lie algebra of $O(V)$, namely
\[ \fo(V)=\{x\in\End(V)\,|\,\langle xv,v\rangle=0,\text{ for all }
v\in V\}, \]
which corresponds to the space of alternating bilinear forms $\Lambda^2(V^*)$.
In matrix terms, $\fsp(V)$ consists of all matrices which are symmetric about the 
skew diagonal (in the sense that the $(i,j)$-entry equals the $(2n+1-j,2n+1-i)$-entry),
and $\fo(V)$ consists of all matrices which are symmetric about the skew diagonal
and have zero entries on the skew diagonal (that is, the $(i,2n+1-i)$-entry is zero
for all $i$). 

As a representation of $Sp(V)$, the quotient
$\fsp(V)/\fo(V)$ is isomorphic to the space $V^\sharp$ of 
functions $f:V\to\F$ which are
Frobenius-semilinear in the sense that 
\[ f(v+w)=f(v)+f(w),\ f(av)=a^2f(v),\text{ for all }a\in\F,\, v,w\in V. \]
The isomorphism sends the coset of $x\in\fsp(V)$ to the function 
$v\mapsto\langle xv,v\rangle$. 
The representation $V^\sharp$ is in turn
isomorphic to the Frobenius twist $V^{(1)}$ (which is
the same as $V$ except that there is a new scalar multiplication $a.v=a^{1/2}v$), via
the map which sends $v\in V^{(1)}$ to the function $\langle v,\cdot\rangle^2$.
Thus we have a short exact sequence
of representations of $Sp(V)$:
\begin{equation} \label{ses1}
\xymatrix{
0 \ar[r] & \fo(V) \ar[r] & \fsp(V) \ar[r]^\pi & V^{(1)} \ar[r] & 0
}
\end{equation}
Indeed, $\fsp(V)$ is the Weyl module for $Sp(V)$
of highest weight $2\varepsilon_1$ (the highest
root), and $V^{(1)}$ is the irreducible module of highest weight $2\varepsilon_1$.
In matrix terms, if $x$ is the matrix $(a_{ij})$, then 
$\pi(x)=\displaystyle\sum_{i=1}^{2n} a_{i,2n+1-i}^{1/2}\,e_i$.

Now the crucial observation of Kato in \cite[Section 4]{kato:deformations} is that
there is a `non-linear splitting' of the short exact sequence \eqref{ses1}.
Namely, we have an $Sp(V)$-equivariant map $s:V\to\fsp(V)$ defined by
\[
s(v)(w)=\langle v,w\rangle v,\text{ for all }
v,w\in V,
\]
such that $\pi(s(v))=v$. In matrix terms, the $(i,j)$-entry of $s(\sum a_ie_i)$ 
is $a_ia_{2n+1-j}$. Note that $s$ does not preserve addition.

As a result, we have an $Sp(V)$-equivariant bijection (denoted $ml$ in 
\cite{kato:deformations})
\begin{equation}
\Psi:V\oplus\fo(V)\to\fsp(V):(v,x)\mapsto s(v)+x.
\end{equation}
Note that $\Psi$ is a morphism of varieties but its inverse 
$\Psi^{-1}:y\mapsto(\pi(y),y-s(\pi(y)))$ is not ($\pi$ is a morphism if its codomain
is $V^{(1)}$, but not if we reinterpret the codomain as $V$).

Recall that the ordinary nilpotent cone of $Sp(V)$ is
\[
\cN(\fsp(V))=\{y\in\fsp(V)\,|\,y\text{ is nilpotent}\}.
\]
Kato's exotic nilpotent cone of $Sp(V)$, in the present context, is
\[
\fN=V\times\cN(\fo(V))=\{(v,x)\in V\oplus\fo(V)\,|\,x\text{ is nilpotent}\}.
\]
Actually Kato defines $\fN$ in \cite{kato:deformations} as the Hilbert nullcone
of the $Sp(V)$-representation $V\oplus\fo(V)$, then proves that $\fN=V\times\cN(\fo(V))$.
\begin{prop}[Kato]
$\Psi$ restricts to an $Sp(V)$-equivariant bijective morphism
$\Psi:\fN\to\cN(\fsp(V))$.
\end{prop}
\begin{proof}
This amounts to saying that for $v\in V$ and $x\in\fo(V)$,
$\Psi(v,x)$ is nilpotent if and only if $x$ is nilpotent. For any $i$, we have
\[ s(v)(x^i v)=\langle v,x^i v\rangle v = \langle x^{\lceil i/2\rceil} v,
x^{\lfloor i/2\rfloor} v\rangle v =0, \]
because $\langle w,w\rangle=\langle xw,w\rangle=0$ for all $w\in V$.
It follows that $\Psi(v,x)$ preserves the subspace $\F[x]v$ spanned by all $x^i v$, and 
induces the same endomorphisms of $\F[x]v$ and $V/\F[x]v$ as $x$ does. The result follows.
\end{proof}

The classification of $Sp(V)$-orbits in $\cN(\fsp(V))$ was found by Hesselink
in \cite{hesselink}, and the
closure ordering on these orbits by Groszer. 
For us it is more convenient to use the
reformulation of their results in terms of bipartitions, due to Spaltenstein. 

Recall the definition of the map $\Phi^C:\cQ_n\to\cP_{2n}^C$ from the previous section.
We need to supplement this by defining the \emph{Hesselink index},
a function $\chi_{\mu;\nu}$ on the set of numbers which occur as parts
of $\Phi^C(\mu;\nu)$, as follows. For every part $2s$ which was left unaltered
from the composition $(2\mu_1,2\nu_1,2\mu_2,2\nu_2,\cdots)$, 
we define $\chi_{\mu;\nu}(2s)=s$; for every
part $s+t$ arising from $2s<2t$, we define $\chi_{\mu;\nu}(s+t)=s$. It is easy
to see that these definitions never conflict.
Observe that $\Phi^C(\varnothing;\nu)$ is the duplicated partition 
$\nu\cup\nu=(\nu_1,\nu_1,\nu_2,\nu_2,\cdots)$,
and that $\chi_{\varnothing;\nu}(\nu_i)=0$ for all $i$.

\begin{thm}[Hesselink, Spaltenstein] \label{thm:spalt-C}
\cite[3.3, 3.6, 4.2]{spaltenstein}
\begin{enumerate}
\item The $Sp(V)$-orbits in $\cN(\fsp(V))$ are in bijection with $\cQ_n$.
For $(\mu;\nu)\in\cQ_n$, the corresponding orbit $\cO_{\mu;\nu}^{C,2}$ consists of those
$y\in\cN(\fsp(V))$ whose Jordan type is $\Phi^C(\mu;\nu)$ and which satisfy
\[
\min\{j\geq 0\,|\,\langle y^{2j+1}w,w\rangle=0,\textup{ for all }w\in\ker(y^m)\}=
\chi_{\mu;\nu}(m),
\]
for all parts $m$ of $\Phi^C(\mu;\nu)$.
\item For $(\rho;\sigma),(\mu;\nu)\in\cQ_n$,
\[
\cO_{\rho;\sigma}^{C,2}\subseteq\overline{\cO_{\mu;\nu}^{C,2}}\Longleftrightarrow
(\rho;\sigma)\leq(\mu;\nu).
\]
\end{enumerate}
\end{thm}

In particular, the $Sp(V)$-orbits in $\cN(\fo(V))$ arise as the special case where
$\mu=\varnothing$: $\cO_{\varnothing;\nu}^{C,2}$ consists of those 
$y\in\cN(\fo(V))$ whose Jordan type is $\nu\cup\nu$. (Each such orbit 
splits into finitely
many $O(V)$-orbits, which are also described in \cite{spaltenstein}.)

Spaltenstein observed in \cite[3.8]{spaltenstein}
that his bijection between $\cQ_n$ and $Sp(V)$-orbits in $\cN(\fsp(V))$ 
would have to be the Springer
correspondence, if the latter could be defined successfully in characteristic $2$ (in
particular, if it could be shown to behave well with respect to induction).
His expectation has been verified by Xue in \cite[Theorem 8.2]{xue:combinatorics}.

Following Kato, we deduce from Theorem \ref{thm:spalt-C} 
that Theorem \ref{thm:exotic} persists in characteristic $2$.
\begin{thm}[Kato] \label{thm:kato}
\cite[Theorem 4.1, Corollary 4.3]{kato:deformations}
\begin{enumerate}
\item The $Sp(V)$-orbits in $\fN$ are in bijection with $\cQ_n$.
For $(\mu;\nu)\in\cQ_n$, the corresponding orbit $\Oo_{\mu;\nu}$ 
is the intersection of $\fN$ with the orbit
$\cO_{\mu\cup\mu;\nu\cup\nu}$
in the enhanced nilpotent cone $V\times\cN(\fgl(V))$. 
\item For $(\rho;\sigma),(\mu;\nu)\in\cQ_n$,
\[
\Oo_{\rho;\sigma}\subseteq\overline{\Oo_{\mu;\nu}}\Longleftrightarrow
(\rho;\sigma)\leq(\mu;\nu).
\]
\end{enumerate}
\end{thm}

\begin{proof}
Since $\Psi$ is bijective and $Sp(V)$-equivariant, the $Sp(V)$-orbits in $\fN$
are the preimages under $\Psi$ of the orbits $\cO_{\mu;\nu}^{C,2}$ in $\cN(\fsp(V))$.
In particular, the number of $Sp(V)$-orbits in $\fN$ is $|\cQ_n|$.
It is clear that each $Sp(V)$-orbit in $\fN$ is contained in a single
$GL(V)$-orbit in the enhanced nilpotent cone $V\times\cN(\fgl(V))$, and the explicit
representative used in the proof of \cite[Theorem 6.1]{ah} shows that for each
$(\mu;\nu)\in\cQ_n$, the
enhanced nilpotent orbit $\cO_{\mu\cup\mu;\nu\cup\nu}$ does intersect $\fN$.
Part (1) follows. Knowing this, the proof of part (2) is the same as in the
characteristic $\neq 2$ case (\cite[Theorem 6.3]{ah}).
\end{proof}

It follows that the definition of type-$B$, type-$C$, and special pieces of $\fN$ given
in Definitions \ref{defn:pieces} and \ref{defn:special} can be applied also to
the characteristic $2$ case.

\begin{rmk}
Kato's terminology is slightly different: his statement
in \cite[Corollary 4.3]{kato:deformations} that different
orbits in $\fN$ have different `$k$-invariants' is equivalent to saying that
they are contained in different $GL(V)$-orbits 
in the enhanced nilpotent cone $V\times\cN(\fgl(V))$; the 
`marked partitions' of \cite[Definitions 3.11, 3.12]{kato:deformations}
are another way of parametrizing these enhanced nilpotent orbits, equivalent
to the bipartitions of $2n$ used above.
\end{rmk}

Kato shows in \cite[Theorem 10.7]{kato:deformations} that the orbit
$\Oo_{\mu;\nu}$ is the one which corresponds to $(\mu;\nu)$
(or rather, the irreducible representation of $W(C_n)$ with that label) under
his exotic Springer correspondence.

Comparing Theorems \ref{thm:spalt-C} and \ref{thm:kato}, it is natural to guess
that the orbit $\cO_{\mu;\nu}^{C,2}$ (specified by Jordan form and Hesselink index) and
the orbit $\Oo_{\mu;\nu}$ (specified by the enhanced nilpotent orbit in which it lies)
correspond under the bijection $\Psi$, and this is indeed the case.

\begin{thm} \label{thm:matchup}
For any $(\mu;\nu)\in\cQ_n$, $\cO_{\mu;\nu}^{C,2}=\Psi(\Oo_{\mu;\nu})$.
\end{thm}
\begin{proof}
This result can be proved by taking an explicit representative $(v,x)\in\Oo_{\mu;\nu}$,
and checking (by a somewhat painful calculation) that $\Psi(v,x)\in\cO_{\mu;\nu}^{C,2}$.
However, the results of \cite{xue:combinatorics} and \cite{kato:deformations} provide
a cleaner proof: since the orbits $\cO_{\mu;\nu}^{C,2}$ and $\Oo_{\mu;\nu}$
correspond to the same irreducible representation of $W(C_n)$ under the respective
Springer correspondences, they must correspond under $\Psi$, by the deformation argument
of \cite[Theorem B]{kato:deformations}. Compare \cite[Theorem D]{kato:deformations}.
\end{proof}

An immediate consequence of the bijection between $\Oo_{\mu;\nu}$ and
$\cO_{\mu;\nu}^{C,2}$ is that they have the same number of $\F_q$-points
for any $q=2^s$; it is
clear from the descriptions of these orbits that they are defined over
$\F_q$, as is the bijection $\Psi$. Moreover, this number of $\F_q$-points
is given by the same polynomial $P_{\mu;\nu}(t)$ as in Proposition
\ref{prop:poly-ex}:
\begin{prop} \label{prop:poly-C2}
For any $(\mu;\nu)\in\cQ_n$ and any $q=2^s$,
\[
|\cO_{\mu;\nu}^{C,2}(\F_q)|=|\Oo_{\mu;\nu}(\F_q)|=P_{\mu;\nu}(q).
\]
\end{prop}
\begin{proof}
Now that Theorem \ref{thm:kato} is known, the description of stabilizers
for the exotic nilpotent cone given in \cite[Proposition 4.5]{kato:exotic} carries
across verbatim to the case of characteristic $2$. So the reductive quotient
of each stabilizer is a product of symplectic groups (Spaltenstein observed the equivalent
fact for $\cN(\fsp(V))$
in \cite[3.9]{spaltenstein}), and the ranks of these symplectic groups
are the same as for the corresponding orbit in characteristic
$\neq 2$. Since the polynomial giving $|Sp_{2k}(\F_q)|$ is independent of characteristic,
the result follows.
\end{proof}
In particular, the dimensions of the orbits $\cO_{\mu;\nu}^{C,2}$ and
$\Oo_{\mu;\nu}$ in characteristic $2$ are the same as that of
$\Oo_{\mu;\nu}$ in characteristic $\neq 2$, namely $2(n^2-b(\mu;\nu))$.

We now define locally closed subvarieties
of $\cN(\fsp(V))$ corresponding to the type-$C$ and special pieces
of the exotic nilpotent cone:
\begin{defn} \label{defn:nilppiece-C}
For any $(\mu;\nu)\in\cQ_n^C$ (in the former case) or $(\mu;\nu)\in\cQ_n^\circ$ 
(in the latter case), define
\[
\cT_{\mu;\nu}^C=\bigdisjunion_{\substack{(\rho;\sigma)\in\cQ_n\\
(\rho;\sigma)^C=(\mu;\nu)}} \cO_{\rho;\sigma}^{C,2}
\quad\text{and}\quad
\cS_{\mu;\nu}^C=\bigdisjunion_{\substack{(\rho;\sigma)\in\cQ_n\\
(\rho;\sigma)^\circ=(\mu;\nu)}} \cO_{\rho;\sigma}^{C,2}.
\]
\end{defn}
By definition, $\cT_{\mu;\nu}^C$ consists of all elements $y\in\cN(\fsp(V))$ whose
Jordan type is $\Phi^C(\mu;\nu)$; note that in characteristic $\neq 2$, the same
definition gives the orbit $\cO_{\mu;\nu}^{C}$. This subvariety $\cT_{\mu;\nu}^C$ is the 
\emph{nilpotent piece} defined by Lusztig and Xue in
\cite[Appendix]{lusztig:unip} (see also \cite{xue:note}). 

\begin{prop}[Lusztig] \label{prop:nilppiece-C}
For any $(\mu;\nu)\in\cQ_n^C$ and any $q=2^s$,
\[ |\cT_{\mu;\nu}^C(\F_q)|=P_{\mu;\nu}^C(q). \]
\end{prop}
\begin{proof}
This is the type-$C$ case of the statement
at the end of \cite[A.6]{lusztig:unip},
that the number of $\F_q$-points in a nilpotent 
piece is given by a polynomial which is
independent of the characteristic. The argument for type $C$
was essentially already present in \cite[3.14]{lusztig:unip1}.
\end{proof}

Combining Propositions \ref{prop:poly-C2} and \ref{prop:nilppiece-C}, one deduces the
$t=2^s$ case, and hence the general case, of the polynomial identity which is
Theorem \ref{thm:main}(2).
\section{Comparing exotic and orthogonal nilpotent cones in characteristic 2}
Continue the assumptions and notation of the previous section. Write $\F e_0$ for
a one-dimensional vector space over $\F$ with basis $e_0$, and set
$\tV=\F e_0\oplus V$. We extend the quadratic form $Q$ to a quadratic form $\tQ$ on $\tV$,
defined by $\tQ(ae_0+v)=a^2+Q(v)$ for $a\in\F$, $v\in V$. The associated
alternating (and hence symmetric) form $\langle\cdot,\cdot\rangle$ on $\tV$ has $\F e_0$
as its radical, and restricts to the same nondegenerate form on $V$ as
in the previous section.

Let $O(\tV)$ be the stabilizer in $GL(\tV)$ of the quadratic form $\tQ$. This is
a connected simple group of type $B_n$ (for $n\geq 2$). Note that $O(\tV)$ preserves
the subspace $\F e_0$ and acts trivially on it, but does not preserve the subspace $V$;
the subgroup preserving $V$ can be identified with $O(V)$.
However, we can make $O(\tV)$ act on $V$ by identifying $V$ with $\tV/\F e_0$.
This gives a homomorphism $\Gamma:O(\tV)\to Sp(V)$
which is a bijective isogeny (not an
isomorphism of algebraic groups). The restriction of $\Gamma$ to $O(V)$ is 
the identity $O(V)\to O(V)$. 
Again, a reference is \cite[Exercises 7.4.7]{springer}.

The Lie algebra of $O(\tV)$ is
\[
\fo(\tV)=\{x\in\End(\tV)\,|\,xe_0=0,\,\langle xv,v\rangle=0,\text{ for all }v\in V\}.
\]
The Lie algebra homomorphism $\gamma$ induced by $\Gamma$ is the surjective
homomorphism $\fo(\tV)\to\fo(V)$ obtained by identifying $V$ with $\tV/\F e_0$ (in 
particular, its image is not the whole of $\fsp(V)$). The kernel of $\gamma$
can be identified with $V$, via the map $\delta$ which sends $v\in V$ to the
endomorphism $w\mapsto \langle v,w\rangle e_0$ of $\tV$. Thus we have a short exact
sequence of representations of $O(\tV)$:
\begin{equation} \label{ses2}
\xymatrix{
0 \ar[r] & V \ar[r]^\delta & \fo(\tV) \ar[r]^\gamma & \fo(V) \ar[r] & 0
}
\end{equation}
In matrix terms, $\gamma$ simply deletes the $0$th row and column (that is, the row and
column corresponding to $e_0$), and $\delta(a_1e_1+\cdots+a_{2n}e_{2n})$ 
is the matrix whose entry in
row $0$ and column $i$ is $a_{2n+1-i}$ for $1\leq i\leq 2n$, 
with all other entries being zero.

There is an obvious section of $\gamma$, namely the embedding
$i:\fo(V)\to\fo(\tV)$ whose image consists of the endomorphisms preserving the
subspace $V$. Note that $i$ is $O(V)$-equivariant but not $O(\tV)$-equivariant.
As a result, we have an $O(V)$-equivariant vector space isomorphism
\begin{equation}
\tPsi:V\oplus\fo(V)\to\fo(\tV):(v,x)\mapsto\delta(v)+i(x).
\end{equation}
This trivial observation allows us to connect the exotic nilpotent cone with 
the nilpotent cone in $\fo(\tV)$.

\begin{prop} \label{prop:nonequivisom}
$\tPsi$ restricts to an $O(V)$-equivariant 
\textup{(}but not $O(\tV)$-equivariant\textup{)} isomorphism of varieties
$\fN\to\cN(\fo(\tV))$.
\end{prop}
\begin{proof}
Let $(v,x)\in V\oplus\fo(V)$.
Since $\F e_0\subseteq\ker(\tPsi(v,x))$, $\tPsi(v,x)$ is nilpotent if and only if
it induces a nilpotent endomorphism of $\tV/\F e_0\cong V$. But the endomorphism it
induces is exactly $x$.
\end{proof}

\begin{rmk}\label{rmk:deformation}
We will not need it, but there is a way to regard $\fN$ as an $O(\tV)$-equivariant 
flat deformation of
$\cN(\fo(\tV))$, in the spirit of \cite{kato:deformations}. Namely, let $O(\tV)$
act on $\fN\times\F$ by the rule
\begin{equation}
g.(v,x,t)=(\Gamma(g)v+t\,\Gamma(g)xc(g^{-1}),\Gamma(g).x,t),
\end{equation}
where $c:O(\tV)\to V$ is the $1$-cocycle defined by
\begin{equation}
\langle c(g),v\rangle=Q(\Gamma(g^{-1})v)^{1/2}+Q(v)^{1/2},\text{ for all }g\in O(\tV),\,
v\in V.
\end{equation}
Then the action of $O(\tV)$ on the zero fibre $\fN\times\{0\}$ is the standard action
on $\fN$ via the isogeny to $Sp(V)$, whereas for $t\neq 0$, we have an $O(\tV)$-equivariant
isomorphism
\begin{equation}
\tPsi_t:\fN\times\{t\}\to\cN(\fo(\tV)):(v,x,t)\mapsto t^{-1}\delta(v)+i(x). 
\end{equation}
The isomorphism $\tPsi$ of Proposition \ref{prop:nonequivisom} is the $t=1$
case of $\tPsi_t$.
\end{rmk}

To recall Spaltenstein's classification of orthogonal nilpotent orbits, we need
some more notation. Let $\cP_{2n}^{\dup}$ denote the set of partitions of
$2n$ which are of the form $\rho\cup\rho$ for some partition $\rho$ of $n$, and let
$\cP_{2n+1}^{\dup}$ denote the set of partitions of
$2n+1$ which after subtracting $1$ from a single part (that is, removing one box from
the Young diagram) become a partition in $\cP_{2n}^{\dup}$.
We define a map 
$\Phi^{B,2}:\cQ_n\to\cP_{2n+1}^{\dup}$ by setting
$\Phi^{B,2}(\mu;\nu)$ to be $(\mu+\nu)\cup(\mu+\nu)$ with $1$
added to the $(2\ell(\nu)+1)$th part (making that part, which was $\mu_{\ell(\nu)+1}$,
now equal to $\mu_{\ell(\nu)+1}+1$).

\begin{exam}
When $n=2$, the map $\Phi^{B,2}:\cQ_2\to\cP_5^{\dup}$ is as follows:
\[
\begin{split}
\Phi^{B,2}((2);\varnothing)&=(32),\
\Phi^{B,2}((1);(1))=(2^21),\
\Phi^{B,2}(\varnothing;(2))=(2^21),\\
\Phi^{B,2}((1^2);\varnothing)&=(21^3),\
\Phi^{B,2}(\varnothing;(1^2))=(1^5). 
\end{split}
\]
\end{exam}

We also define a function $\tchi_{\mu;\nu}$ on the set of numbers which occur
as parts of $\Phi^{B,2}(\mu;\nu)$, by the (clearly self-consistent) rules
\[ \tchi_{\mu;\nu}(\mu_{\ell(\nu)+1}+1)=\mu_{\ell(\nu)+1}+1,\
\tchi_{\mu;\nu}(\mu_i+\nu_i)=\begin{cases}
\mu_i+1,&\text{ if $i\leq\ell(\nu)$,}\\
\mu_i,&\text{ if $i>\ell(\nu)$.}
\end{cases} \]

Finally, we define a new sub-poset of $\cQ_n$,
containing the sub-poset $\cQ_n^B$:
\[
\cQ_n^{B,2}=\{(\mu;\nu)\in\cQ_n\,|\,\mu_i\geq\nu_i-2,\text{ for all }i\}.
\]

\begin{thm}[Hesselink, Spaltenstein] \label{thm:spalt-B}
\cite[3.2, 4.2]{spaltenstein}
\begin{enumerate}
\item
The $O(\tV)$-orbits in $\cN(\fo(\tV))$ are in bijection with $\cQ_n^{B,2}$.
For $(\mu;\nu)\in\cQ_n^{B,2}$, the corresponding orbit $\cO_{\mu;\nu}^{B,2}$ consists
of those $y\in\cN(\fo(\tV))$ whose Jordan type is $\Phi^{B,2}(\mu;\nu)$ and which
satisfy 
\[
\min\{j\geq 0\,|\,\tQ(y^j w)=0,\textup{ for all }w\in\ker(y^m)\}=
\tchi_{\mu;\nu}(m),
\]
for all parts $m$ of $\Phi^{B,2}(\mu;\nu)$.
\item For $(\rho;\sigma),(\mu;\nu)\in\cQ_n^{B,2}$,
\[
\cO_{\rho;\sigma}^{B,2}\subseteq\overline{\cO_{\mu;\nu}^{B,2}}\Longleftrightarrow
(\rho;\sigma)\leq(\mu;\nu).
\]
\end{enumerate}
\end{thm}

In \cite[Theorem 9.2]{xue:combinatorics}, it is proved that the orbit 
$\cO_{\mu;\nu}^{B,2}$, with the trivial local system, corresponds to the
irreducible representation of $W(B_n)$ with character $\chi^{\mu;\nu}$ under the
Springer correspondence.

The preimages under $\tPsi$ of the orbits $\cO_{\mu;\nu}^{B,2}$ are some 
$O(V)$-stable locally closed subvarieties of the exotic nilpotent cone $\fN$,
which are not, in general, unions of the orbits $\Oo_{\rho;\sigma}$ described
in Theorem \ref{thm:kato} (that is, they are not $Sp(V)$-stable). To relate these
two types of pieces of $\fN$, it is convenient to collect together various orbits
according to the Jordan type of the nilpotent endomorphism. 

For $(\rho;\sigma),(\mu;\nu)\in\cQ_n$, we write
\[
(\rho;\sigma)\preceq(\mu;\nu)\text{ to mean that }(\rho;\sigma)\leq(\mu;\nu)
\text{ and }\rho+\sigma=\mu+\nu.
\]
Note that under the assumption $\rho+\sigma=\mu+\nu$, the condition
$(\rho;\sigma)\leq(\mu;\nu)$ is equivalent to the condition 
that $\rho_i\leq\mu_i$ for all $i$.
\begin{defn}
For any $(\mu;\nu)\in\cQ_n$ (in the former case) or 
$(\mu;\nu)\in\cQ_n^{B,2}$ (in the latter case), define
\[
\Ee_{\mu;\nu}=\bigdisjunion_{\substack{(\rho;\sigma)\in\cQ_n\\(\rho;\sigma)\preceq(\mu;\nu)}}
\Oo_{\rho;\sigma}\quad\text{and}\quad
\cE_{\mu;\nu}=\bigdisjunion_{\substack{(\rho;\sigma)\in\cQ_n^{B,2}\\
(\rho;\sigma)\preceq(\mu;\nu)}}\cO_{\rho;\sigma}^{B,2}.
\]
\end{defn}
It follows from Theorem \ref{thm:kato}(2) that $\Ee_{\mu;\nu}$ is an open
subvariety of the closure $\overline{\Oo_{\mu;\nu}}$, and from 
Theorem \ref{thm:spalt-B}(2) that $\cE_{\mu;\nu}$ is an open subvariety
of the closure $\overline{\cO_{\mu;\nu}^{B,2}}$.

\begin{prop} \label{prop:vectorbundle}
Let $(\mu;\nu)\in\cQ_n$.
\begin{enumerate}
\item
$\Ee_{\mu;\nu}$ consists of those $(v,x)\in\fN$ which satisfy the following conditions:
\begin{enumerate}
\item $x$ has Jordan type $(\mu+\nu)\cup(\mu+\nu)$.
\item For all $i$ and all $u\in\ker(x^{\mu_i+\nu_i})$, we have
$\langle v,x^{\mu_i}u\rangle=0$.
\end{enumerate}
\item
The projection $\Ee_{\mu;\nu}\to\cO_{\varnothing;\mu+\nu}^{C,2}:(v,x)\mapsto x$
is a vector bundle of rank $2|\mu|$.
\end{enumerate}
\end{prop}

\begin{proof}
From Theorem \ref{thm:kato}, it follows that $\Ee_{\mu;\nu}$ consists of those
$(v,x)\in\fN$ such that $x$ has Jordan type $(\mu+\nu)\cup(\mu+\nu)$ and $v$ belongs to
the subspace $\sum_i x^{\nu_i}(\ker(x^{\mu_i+\nu_i}))$. It is easy to see that this
subspace has dimension $2|\mu|$, so we need only prove that it is the perpendicular
subspace to $\sum_i x^{\mu_i}(\ker(x^{\mu_i+\nu_i}))$. Since the dimensions are
correct, it suffices to show that for all $i,j$ and all $u_i\in\ker(x^{\mu_i+\nu_i})$,
$u_j\in\ker(x^{\mu_j+\nu_j})$, we have $\langle x^{\nu_i}u_i,x^{\mu_j}u_j\rangle=0$.
But if $i\geq j$, then 
$\langle x^{\nu_i}u_i,x^{\mu_j}u_j\rangle=\langle x^{\mu_j+\nu_i}u_i,u_j\rangle=0$,
since $\mu_j\geq\mu_i$; and if $i<j$, then
$\langle x^{\nu_i}u_i,x^{\mu_j}u_j\rangle=\langle u_i,x^{\mu_j+\nu_i}u_j\rangle=0$,
since $\nu_i\geq\nu_j$. The proof is finished.
\end{proof}

\begin{prop} \label{prop:jordan}
Let $(\rho;\sigma)\in\cQ_n$ and $(v,x)\in\fN$. Then $\tPsi(v,x)$ has Jordan type
$\Phi^{B,2}(\rho;\sigma)$ if and only if the following conditions hold:
$x$ has Jordan type $(\rho+\sigma)\cup(\rho+\sigma)$, and 
\begin{equation} \label{mineqn}
\min\{j\geq 1\,|\,\langle v,x^{j-1}(\ker(x^j))\rangle=0\}=\rho_{\ell(\sigma)+1}+1.
\end{equation}
\end{prop}
\begin{proof}
Set $y=\tPsi(v,x)$.
Since $x$ is the endomorphism induced by $y$
on $\tV/\F e_0\cong V$, the Jordan type of $x$ is obtained from that of $y$ by removing
a single box, which must be the unique box whose removal leaves a partition in
$\cP_{2n}^\dup$. So $x$ has Jordan type $(\rho+\sigma)\cup(\rho+\sigma)$
if and only if $y$ has Jordan type obtained from this by adding a box.
We assume this henceforth; what remains is to prove that the added box is
in column $\rho_{\ell(\sigma)+1}+1$ if and only if \eqref{mineqn} holds.

The column number of the added box can be expressed as
\[ \min\{j\geq 1\,|\,\dim\ker(y^j)=1+\dim\ker(x^j)\}. \]
Now for any $a\in\F$ and $u\in V$, we have by definition
\begin{equation} \label{yjeqn}
y^j(ae_0+u)=\langle v,x^{j-1}u\rangle e_0+x^j u,\text{ for all }j\geq 1.
\end{equation}
Hence $\ker(y^j)\subseteq \F e_0\oplus\ker(x^j)$, with equality if and only if 
$\langle v,x^{j-1}(\ker(x^j))\rangle=0$. 
So
\[
\min\{j\geq 1\,|\,\dim\ker(y^j)=1+\dim\ker(x^j)\}=
\min\{j\geq 1\,|\,\langle v,x^{j-1}(\ker(x^j))\rangle=0\},
\]
and the proof is finished.
\end{proof}

\begin{rmk}
It is easy to see that the explicit representative
$(v_{\rho;\sigma},x_{\rho;\sigma})\in\Oo_{\rho;\sigma}$ defined in the proof of
\cite[Theorem 6.1]{ah}
satisfies the conditions in Proposition \ref{prop:jordan},
so $\tPsi(v_{\rho;\sigma},x_{\rho;\sigma})$ has Jordan type
$\Phi^{B,2}(\rho;\sigma)$. In fact, one can check that for $(\mu;\nu)\in\cQ_n^{B,2}$,
a representative of the orbit $\cO_{\mu;\nu}^{B,2}$ is
$\tPsi(v_{\mu;\nu},x_{\mu;\nu})$. 
\end{rmk}

\begin{prop} \label{prop:affinebundle}
Let $(\mu;\nu)\in\cQ_n^{B,2}$.
\begin{enumerate}
\item
$\tPsi^{-1}(\cE_{\mu;\nu})$ consists of 
those $(v,x)\in\fN$ which satisfy the following conditions:
\begin{enumerate}
\item $x$ has Jordan type $(\mu+\nu)\cup(\mu+\nu)$.
\item For all $i$ and all $u\in\ker(x^{\mu_i+\nu_i})$, we have
$\langle v,x^{\mu_i}u\rangle=Q(x^{\mu_i+1}u)^{1/2}$.
\end{enumerate}
\item
The projection $\tPsi^{-1}(\cE_{\mu;\nu})
\to\cO_{\varnothing;\mu+\nu}^{C,2}:(v,x)\mapsto x$
is a fibre bundle in the Zariski topology, where the fibres are homeomorphic
to $\F^{2|\mu|}$.
\end{enumerate}
\end{prop}

\begin{proof}
Let $(v,x)\in\fN$, and set $y=\tPsi(v,x)$. Let $(\rho;\sigma)\in\cQ_n^{B,2}$ be
the bipartition such that $y\in\cO_{\rho;\sigma}^{B,2}$. To prove (1),
we must show that $(v,x)$ satisfies conditions (a) and (b) if and only if
$(\rho;\sigma)\preceq(\mu;\nu)$. But by
Proposition \ref{prop:jordan}, $x$ has Jordan type $(\rho+\sigma)\cup(\rho+\sigma)$,
so condition (a) is equivalent to $\rho+\sigma=\mu+\nu$. We assume this henceforth;
what remains is to show that condition (b) is equivalent to the statement that
$\rho_i\leq\mu_i$ for all $i$.

We first suppose that $\rho_i\leq\mu_i$ for all $i$ and deduce condition (b).
Note that $\ell(\sigma)\geq\ell(\nu)$.
Since condition (b) is trivially true for $i>\ell(\nu)$, we assume that
$i\leq\ell(\nu)$. By the definition of $\cO_{\rho;\sigma}^{B,2}$, we have
$\tQ(y^{\rho_i+1}w)=0$ for all $w\in\ker(y^{\mu_i+\nu_i})$.
By the basic theory of form modules, this equation implies the equation
$\tQ(y^s w)=0$ for any $s\geq\rho_i+1$, in particular for $s=\mu_i+1$.
From the proof of
Proposition \ref{prop:jordan} and the fact that $\mu_i+\nu_i\geq
\mu_{\ell(\nu)+1}+1\geq\rho_{\ell(\sigma)+1}+1$, we see that 
$\ker(y^{\mu_i+\nu_i})=\F e_0\oplus\ker(x^{\mu_i+\nu_i})$. So
for all $u\in\ker(x^{\mu_i+\nu_i})$, we have
$\tQ(y^{\mu_i+1}u)=0$; using \eqref{yjeqn} and the definition of $\tQ$, this becomes
the desired equation $\langle v,x^{\mu_i}u\rangle=Q(x^{\mu_i+1}u)^{1/2}$.

For the converse, we suppose that $\rho_i\geq\mu_i+1$ for some $i$, and show that
condition (b) fails for this particular $i$, which obviously satisfies $i\leq\ell(\nu)$.
By the definition of $\cO_{\rho;\sigma}^{B,2}$, we have
$\tQ(y^{\rho_i}w)\neq 0$, and therefore also $\tQ(y^{\mu_i+1}w)\neq 0$,
for some $w\in\ker(y^{\mu_i+\nu_i})$. Using \eqref{yjeqn} again, we see that
this implies that $\langle v,x^{\mu_i}u\rangle\neq Q(x^{\mu_i+1}u)^{1/2}$ for some
$u\in\ker(x^{\mu_i+\nu_i})$. Part (1) is now proved.

It follows from the description in part (1) that $\tPsi^{-1}(\cE_{\mu;\nu})
\to\cO_{\varnothing;\mu+\nu}^{C,2}$ is a fibre bundle in the Zariski topology,
although it cannot be locally trivialized in the category of varieties, because of
the non-regular functions $x\mapsto Q(x^{\mu_i+1}u)^{1/2}$ involved. The fibre over
$x$ is an affine subspace of $V$ based on the vector subspace perpendicular to
$\sum_i x^{\mu_i}(\ker(x^{\mu_i+\nu_i}))$. As seen in Proposition 
\ref{prop:vectorbundle}, this vector space has dimension $2|\mu|$.
\end{proof}

\begin{rmk} \label{rmk:flatfamily}
If we replaced $\tPsi^{-1}(\cE_{\mu;\nu})$ with $\tPsi_t^{-1}(\cE_{\mu;\nu})$ where
$\tPsi_t$ is as in Remark \ref{rmk:deformation}, the equation in Proposition
\ref{prop:affinebundle}(1)(b) would become $\langle v,x^{\mu_i}u\rangle=
t\,Q(x^{\mu_i+1}u)^{1/2}$. Sending $t$ to $0$, we get the variety
$\Ee_{\mu;\nu}$ described in Proposition \ref{prop:vectorbundle}.
So we have a flat family whose zero fibre is $\Ee_{\mu;\nu}$ and whose general fibre
is isomorphic to $\cE_{\mu;\nu}$.
\end{rmk}

It is clear from the descriptions of the orbits $\cO_{\mu;\nu}^{B,2}$ that they are 
defined over $\F_q$ for any $q=2^s$, as are the isomorphism $\tPsi$ and the
bundle projections involved in Propositions \ref{prop:vectorbundle}
and \ref{prop:affinebundle}. Of course, the count of $\F_q$-points makes no
distinction between the true vector bundle of Proposition \ref{prop:vectorbundle}
and the `Zariski-topological vector bundle' of Proposition \ref{prop:affinebundle}. 
Recalling Proposition \ref{prop:poly-C2}, we deduce:

\begin{cor} \label{cor:bundle}
For any $(\mu;\nu)\in\cQ_n^{B,2}$ and any $q=2^s$,
\[
|\cE_{\mu;\nu}(\F_q)|=|\Ee_{\mu;\nu}(\F_q)|=
\sum_{\substack{(\rho;\sigma)\in\cQ_n\\(\rho;\sigma)\preceq(\mu;\nu)}}
P_{\rho;\sigma}(q).
\]
\end{cor}

We now want to derive a formula for $|\cO_{\mu;\nu}^{B,2}(\F_q)|$. 
(In general, $\cO_{\mu;\nu}^{B,2}(\F_q)$
splits into a number of orbits for the finite group $O(\tV)(\F_q)$, which are
described in \cite{xue}; we do not consider these here.) We need one further
combinatorial result and definition.

\begin{prop} 
For $(\rho;\sigma)\in\cQ_n$, there is a unique $(\rho;\sigma)^\sim\in\cQ_n^{B,2}$
satisfying:
\begin{enumerate}
\item $(\rho;\sigma)\preceq(\rho;\sigma)^\sim$, and
\item $(\rho;\sigma)^\sim\leq(\tau;\upsilon)$
for any $(\tau;\upsilon)\in\cQ_n^{B,2}$ such that 
$(\rho;\sigma)\leq(\tau;\upsilon)$.
\end{enumerate}
\end{prop}

\begin{proof}
We construct $(\rho;\sigma)^\sim$ by replacing
$\rho_i,\sigma_i$ with $\lceil\frac{\rho_i+\sigma_i}{2}\rceil-1,
\lfloor\frac{\rho_i+\sigma_i}{2}\rfloor+1$ whenever $\rho_i<\sigma_i-2$.
Properties (1) and (2) are easily checked, and uniqueness follows formally.
\end{proof}

\begin{defn} \label{defn:pieces-B2}
For any $(\mu;\nu)\in\cQ_n^{B,2}$,
we define the associated \emph{type-$(B,2)$ piece} of the exotic nilpotent cone to be:
\[
\Tt_{\mu;\nu}^{B,2}=\bigdisjunion_{\substack{(\rho;\sigma)\in\cQ_n\\
(\rho;\sigma)^\sim=(\mu;\nu)}}
\Oo_{\rho;\sigma}
\;=\;\overline{\Oo_{\mu;\nu}}\,\setminus
\bigcup_{\substack{(\tau;\upsilon)\in\cQ_n^{B,2}\\(\tau;\upsilon)<(\mu;\nu)}}
\overline{\Oo_{\tau;\upsilon}}.
\]
\end{defn}

\begin{thm} \label{thm:poly-B2}
For any $(\mu;\nu)\in\cQ_n^{B,2}$ and any $q=2^s$,
\[
|\cO_{\mu;\nu}^{B,2}(\F_q)|=|\Tt_{\mu;\nu}^{B,2}(\F_q)|=
\sum_{\substack{(\rho;\sigma)\in\cQ_n\\
(\rho;\sigma)^\sim=(\mu;\nu)}}
P_{\rho;\sigma}(q).
\]
\end{thm}

\begin{proof}
By definition, we have
\begin{equation}
\cE_{\mu;\nu}=\bigdisjunion_{\substack{(\rho;\sigma)\in\cQ_n^{B,2}\\
(\rho;\sigma)\preceq(\mu;\nu)}}\cO_{\rho;\sigma}^{B,2}
\quad\text{and}\quad
\Ee_{\mu;\nu}=\bigdisjunion_{\substack{(\rho;\sigma)\in\cQ_n^{B,2}\\
(\rho;\sigma)\preceq(\mu;\nu)}}\Tt_{\rho;\sigma}^{B,2}.
\end{equation}
So the result follows by poset induction from Corollary \ref{cor:bundle}.
\end{proof}

In particular, we get the usual dimension formula 
$\dim\cO_{\mu;\nu}^{B,2}=2(n^2-b(\mu;\nu))$.
Note that our argument has produced not only the equinumeracy result stated in
Theorem \ref{thm:poly-B2}, but also a 
relationship between the actual varieties
$\cO_{\mu;\nu}^{B,2}$ and $\Tt_{\mu;\nu}^{B,2}$: from Propositions \ref{prop:vectorbundle}
and \ref{prop:affinebundle}, it follows that they are both fibre bundles in the
Zariski topology over the
same base $\cO_{\varnothing;\mu+\nu}^{C,2}$, with isomorphic (smooth) fibres.

\begin{rmk} \label{rmk:relationship}
Arguing as in Remark \ref{rmk:flatfamily}, we have a flat family whose zero fibre is
$\Tt_{\mu;\nu}^{B,2}$ and whose general fibre is isomorphic to $\cO_{\mu;\nu}^{B,2}$.
So the relationship between orbits in the nilpotent cones of types $B_n$ and $C_n$ in
characteristic $2$ may be expressed as follows: each orbit $\cO_{\mu;\nu}^{B,2}$ in
$\cN(\fo(\tV))$
has a deformation $\Tt_{\mu;\nu}^{B,2}$ which is in turn in bijection (via $\Psi$) with
a union of orbits in $\cN(\fsp(V))$, namely
\[
\cT^{C;B,2}_{\mu;\nu} = \bigdisjunion_{\substack{(\rho;\sigma)\in\cQ_n\\
(\rho;\sigma)^\sim=(\mu;\nu)}}
\cO^{C,2}_{\rho;\sigma}
\;=\;\overline{\cO^{C,2}_{\mu;\nu}}\,\setminus
\bigcup_{\substack{(\tau;\upsilon)\in\cQ_n^{B,2}\\(\tau;\upsilon)<(\mu; 
\nu)}}
\overline{\cO^{C,2}_{\tau;\upsilon}}.
\]
It is clear that $((\rho;\sigma)^\sim)^C=(\rho;\sigma)^C$, so
each $\cT^{C;B,2}_{\mu;\nu}$ is a locally closed subvariety of the nilpotent piece
$\cT^C_{(\mu;\nu)^C}$ defined in Definition \ref{defn:nilppiece-C}.
\end{rmk}

We now define locally closed subvarieties
of $\cN(\fo(\tV))$ corresponding to the type-$B$ and special pieces
of the exotic nilpotent cone:
\begin{defn} \label{defn:nilppiece-B}
For any $(\mu;\nu)\in\cQ_n^B$ (in the former case) or $(\mu;\nu)\in\cQ_n^\circ$ 
(in the latter case), define
\[
\cT_{\mu;\nu}^B=\bigdisjunion_{\substack{(\rho;\sigma)\in\cQ_n^{B,2}\\
(\rho;\sigma)^B=(\mu;\nu)}} \cO_{\rho;\sigma}^{B,2}
\quad\text{and}\quad
\cS_{\mu;\nu}^B=\bigdisjunion_{\substack{(\rho;\sigma)\in\cQ_n^{B,2}\\
(\rho;\sigma)^\circ=(\mu;\nu)}} \cO_{\rho;\sigma}^{B,2}.
\]
\end{defn}

It is clear that $((\rho;\sigma)^\sim)^B=(\rho;\sigma)^B$, so 
Theorem \ref{thm:poly-B2} implies the following formula for the number of
$\F_q$-points of $\cT_{\mu;\nu}^{B}$:

\begin{cor} \label{cor:poly-B2}
For any $(\mu;\nu)\in\cQ_n^{B}$ and any $q=2^s$,
\[
|\cT_{\mu;\nu}^{B}(\F_q)|=
\sum_{\substack{(\rho;\sigma)\in\cQ_n\\
(\rho;\sigma)^B=(\mu;\nu)}}
P_{\rho;\sigma}(q).
\]
\end{cor}

Unlike the situation in type $C$, the elements of $\cT_{\mu;\nu}^B$ do not
all have the same Jordan type. However, by \cite[4.3]{xue:note}, the piece
$\cT_{\mu;\nu}^B$, which we have defined using the combinatorial map $\Phi^B$, 
coincides with the nilpotent piece
defined by Lusztig and Xue in \cite[Appendix]{lusztig:unip} using
gradings of the Jacobson--Morozov kind. Hence we have:

\begin{prop}[Lusztig, Xue] \label{prop:bpc-count}
For any $(\mu;\nu)\in\cQ_n^B$ and any $q=2^s$,
\[ |\cT_{\mu;\nu}^B(\F_q)|=P_{\mu;\nu}^B(q). \]
\end{prop}

\begin{proof}
This is the type-$B$ case of the statement
at the end of \cite[A.6]{lusztig:unip},
that the number of $\F_q$-points in a nilpotent
piece is given by a polynomial which is independent of the
characteristic.
\end{proof}

Combining Corollary \ref{cor:poly-B2} and Proposition 
\ref{prop:bpc-count}, one deduces the
$t=2^s$ case, and hence the general case, of the polynomial identity which is
Theorem \ref{thm:main}(1).
\section{Smoothness and rational smoothness of pieces}

As mentioned before, the nilpotent pieces $\cT_{\mu;\nu}^C$ and
$\cT_{\mu;\nu}^B$ defined in Definitions \ref{defn:nilppiece-C} and \ref{defn:nilppiece-B}
are given an alternative characterization in \cite[Appendix]{lusztig:unip}, using
gradings of Jacobson--Morozov type. In this approach, each nilpotent
piece is seen to be isomorphic to 
$G \times^P \mathfrak{s}^\Delta$, where $G$ is the appropriate
classical group, $P$ is a parabolic subgroup, and $\mathfrak{s}^\Delta$ is an open subset
of a vector space; hence it is smooth.

We aim now to give an analogous proof of smoothness of the
pieces $\Tt_{\mu;\nu}^C$ and $\Tt_{\mu;\nu}^B$ in the exotic nilpotent cone. This
argument works in any characteristic, so 
let $(V,\langle\cdot,\cdot\rangle)$ be a $2n$-dimensional 
symplectic vector space over $\F$, whose
characteristic is now arbitrary. Recall that
\[
\begin{split}
S&=\{x\in\End(V)\,|\,\langle xw,w\rangle =0,\text{ for all }w\in V\},\\
\fN&=\{(v,x)\in V\times S\,|\,x\text{ nilpotent}\},
\end{split}
\]
and that the type-$C$ and type-$B$ pieces of $\fN$ were defined in 
Definition \ref{defn:pieces} (which applies equally well in characteristic $2$).

We study the type-$C$ pieces first.
\begin{defn} \label{defn:lambda-filt}
For any $\lambda\in\cP_{2n}^C$, a \emph{$\lambda$-filtration} of $V$ is a $\Z$-filtration
$(V_{\geq a})_{a\in\Z}$ of $V$, with $V_{\geq a}\subseteq V_{\geq a-1}$ for all $a$,
such that 
\begin{enumerate}
\item
$V_{\geq 1-a}=(V_{\geq a})^\perp$ for all $a$, and 
\item
$\dim V_{\geq a}=\sum_{i\geq 1}\max\{\lceil\frac{\lambda_i-a}{2}\rceil,0\}$, for all
$a\geq 1$.
\end{enumerate}
\end{defn}
The reason for the notation is that such a filtration 
would arise from a Jacobson--Morozov grading $(V_a)_{a\in\Z}$ associated to a
nilpotent element of $\fsp(V)$ of Jordan type $\lambda$, via the definition
$V_{\geq a}=\bigoplus_{b\geq a}V_b$.

Observe that condition (2) implies $V_{\geq a}=0$ for $a\geq\lambda_1$, and this together
with condition (1) implies
$V_{\geq a}=V$ for $a\leq 1-\lambda_1$, so a $\lambda$-filtration can equivalently
be expressed as a partial flag:
\[ 0=V_{\geq\lambda_1}\subseteq V_{\geq\lambda_1-1}\subseteq V_{\geq \lambda_1-2}
\subseteq\cdots\subseteq 
V_{\geq 3-\lambda_1}\subseteq V_{\geq 2-\lambda_1} \subseteq V_{\geq 1-\lambda_1}=V. \]
Moreover, the dimensions of $V_{\geq a}$ for all $a\leq 0$ are uniquely determined by
conditions (1) and (2), and the partial flag is isotropic by condition (1). 
So the set of all $\lambda$-filtrations can be identified
with the partial flag variety $Sp(V)/P$, where $P$ is the parabolic subgroup of $Sp(V)$
which stabilizes a particular $\lambda$-filtration.

\begin{defn} \label{defn:C-adapted}
For $(v,x)\in\fN$, we say that a $\lambda$-filtration $(V_{\geq a})$ is 
\emph{$C$-adapted} to $(v,x)$ if $v\in V_{\geq 1}$ and 
$x(V_{\geq a})\subseteq V_{\geq a+2}$ for all $a$.
\end{defn}

Our aim is to prove that if $(v,x)$ belongs to the type-$C$ piece 
$\Tt_{\hPhi^C(\lambda)}^C$,
then there is a unique $\lambda$-filtration which is $C$-adapted to $(v,x)$.
We will give a recursive construction of this filtration, 
inspired by \cite[2.10]{xue:note}. 
For any partition $\lambda$, we write $k(\lambda)$ for the
multiplicity of the largest part $\lambda_1$ of $\lambda$ (we interpret $k(\varnothing)$ 
as $\infty$).

\begin{prop} \label{prop:filt}
Let $\lambda\in\cP_{2n}^C$, and assume $\lambda_1\geq 2$. Let 
$\lambda'\in\cP_{2(n-k(\lambda))}^C$ be the partition obtained from $\lambda$ by
replacing all parts equal to $\lambda_1$ with $\lambda_1-2$,
and re-ordering parts if necessary so that these come after any parts equal to 
$\lambda_1-1$. 
\begin{enumerate}
\item 
If $(V_{\geq a})$ is a $\lambda$-filtration of $V$, then 
$\dim V_{\geq\lambda_1-1}=\dim V/V_{\geq 2-\lambda_1}=k(\lambda)$.
We can define a $\lambda'$-filtration $(V_{\geq a}')$ of the symplectic vector space
$V'=V_{\geq 2-\lambda_1}/V_{\geq\lambda_1-1}$ by the rule
\[
V_{\geq a}'=\begin{cases}
0,&\text{ if $a\geq\lambda_1$,}\\
V_{\geq a}/V_{\geq\lambda_1-1}, &\text{ if $2-\lambda_1\leq a\leq\lambda_1-1$,}\\
V',&\text{ if $a\leq 1-\lambda_1$.}
\end{cases}
\]
\item 
Conversely, suppose that $W$ is a $k(\lambda)$-dimensional isotropic subspace of $V$, and
we have a $\lambda'$-filtration $(V_{\geq a}')$ of $V'=W^\perp/W$. Then we can define
a $\lambda$-filtration $(V_{\geq a})$ of $V$ by the rule
\[
V_{\geq a}=\begin{cases}
0,&\text{ if $a\geq\lambda_1$,}\\
\varphi^{-1}(V_{\geq a}'), &\text{ if $2-\lambda_1\leq a\leq\lambda_1-1$,}\\
V,&\text{ if $a\leq 1-\lambda_1$.}
\end{cases}
\]
where $\varphi:W^\perp\to V'$ is the canonical projection.
\end{enumerate}
The constructions in \textup{(1)} and \textup{(2)} are inverse to each other, so specifying
a $\lambda$-filtration of $V$ is the same as specifying a 
$k(\lambda)$-dimensional isotropic subspace $W$ of $V$ and a $\lambda'$-filtration
of $W^\perp/W$.
\end{prop}
\begin{proof}
This is all easily checked using the definition of $\lambda$-filtration.
\end{proof}

\begin{prop} \label{prop:recursion}
Let $\lambda\in\cP_{2n}^C$ with $\lambda_1\geq 2$, and define
$\lambda'$ as in Proposition \ref{prop:filt}.
Take any
$(v,x)\in\Tt_{\hPhi^C(\lambda)}^C$, and set
\[ W_{(v,x)}^\lambda=\begin{cases}
\im(x^{\lambda_1-1}),&\text{ if $\lambda_1$ is odd,}\\
\F x^{\frac{\lambda_1}{2}-1}v+\im(x^{\lambda_1-1}),&\text{ if $\lambda_1$ is even.}
\end{cases} \]
\begin{enumerate}
\item $W_{(v,x)}^\lambda$ is $k(\lambda)$-dimensional, isotropic, and 
contained in $\ker(x)$.
\item If $(V_{\geq a})$ is a $\lambda$-filtration which is $C$-adapted to
$(v,x)$, then $V_{\geq\lambda_1-1}=W_{(v,x)}^\lambda$.
\item In the exotic nilpotent cone for the vector space
$(W_{(v,x)}^\lambda)^\perp/W_{(v,x)}^\lambda$, the induced pair 
$(v+W_{(v,x)}^\lambda,x|_{(W_{(v,x)}^\lambda)^\perp/W_{(v,x)}^\lambda})$ 
lies in
$\Tt_{\hPhi^C(\lambda')}^C$.
\end{enumerate}
\end{prop}

\begin{proof}
Let $(\rho;\sigma)\in\cQ_n$ be such that $(v,x)\in\Oo_{\rho;\sigma}$;
by assumption, $\Phi^C(\rho;\sigma)=\lambda$. Once part (1) is proved,
it makes sense to speak of the exotic nilpotent cone for
$(W_{(v,x)}^\lambda)^\perp/W_{(v,x)}^\lambda$, and
$(v',x')=(v+W_{(v,x)}^\lambda,x|_{(W_{(v,x)}^\lambda)^\perp/W_{(v,x)}^\lambda})$ 
is a well-defined element of that cone, so
we can let $(\rho';\sigma')\in\cQ_{n-k(\lambda)}$ be such that
$(v',x')\in\Oo_{\rho';\sigma'}$.
Then part (3) is equivalent to the statement that $\Phi^C(\rho';\sigma')=\lambda'$.
We divide the proof into three cases depending on $k(\lambda)$.

\noindent
\textbf{Case 1:} $k(\lambda)=1$. 
By the definition of $\Phi^C$, we must have $\lambda_1=2m$ where $m=\rho_1>\sigma_1$.
Thus $(\rho+\sigma)_1\leq 2m-1$, which implies that
$\im(x^{2m-1})=0$, so $W_{(v,x)}^\lambda=\F x^{m-1}v$. Since $m=\rho_1=\dim\F[x]v$,
$x^{m-1}v$ is nonzero, so $W_{(v,x)}^\lambda$ is $1$-dimensional, and clearly
isotropic; moreover, $x^m v=0$, so $W_{(v,x)}^\lambda$ is contained in $\ker(x)$, 
proving (1).
If $(V_{\geq a})$ is a $\lambda$-filtration which is $C$-adapted to
$(v,x)$, then by definition we have $x^{m-1}v\in V_{\geq 2m-1}$
and $\dim V_{\geq 2m-1}=1$,
proving (2). For part (3), it is easy to see 
using explicit orbit representatives that
\begin{equation*}
\rho_i'=\begin{cases}
m-1&\text{ if $i\leq k(\rho)$,}\\
\rho_i&\text{ if $i\geq k(\rho)+1$,}
\end{cases}
\quad\text{and}\quad
\sigma_i'=\begin{cases}
\sigma_i+1&\text{ if $i\leq k(\rho)-1$,}\\
\sigma_i&\text{ if $i\geq k(\rho)$.}
\end{cases}
\end{equation*}
If $\sigma_1\leq m-2$, then 
$\rho_1'\geq\sigma_1'$, and the only difference between $\lambda$ and
$\Phi^C(\rho';\sigma')$ is that the latter begins with $2m-2$ rather than $2m$, so
$\Phi^C(\rho';\sigma')=\lambda'$ as required. If $\sigma_1=m-1$, then $\rho_i'<\sigma_i'$
for all $i\leq\ell$, where $\ell=\min\{k(\rho)-1,k(\sigma)\}$. 
So $\Phi^C(\rho';\sigma')$ begins
$(2m-1,2m-1,\cdots,2m-1,2m-1,2m-2,\cdots)$, with $2\ell$ copies of $2m-1$, where the
remaining parts are the same as those of $\lambda$; again, this is exactly $\lambda'$.

\noindent
\textbf{Case 2:} $k(\lambda)=2c$ for some $c\geq 1$. 
By the definition of $\Phi^C$, we must have
$\rho_1+\sigma_1=\lambda_1$, and either $k(\rho)=c$, $k(\sigma)\geq c$, and
$\rho_1\leq\sigma_1$, or
$k(\rho)>c$, $k(\sigma)=c$, and $\rho_1<\sigma_1$.
Hence $(\rho+\sigma)_1=\lambda_1$, and $k(\rho+\sigma)=c$. 
It follows that $\im(x^{\lambda_1-1})$ is
$2c$-dimensional and contained in $\ker(x)$; also, it is isotropic, since its
perpendicular subspace is $\ker(x^{\lambda_1-1})$ and $\lambda_1\geq 2$.
If $\lambda_1$ is even, then either $x^{\frac{\lambda_1}{2}-1}v=0$
(if $\rho_1<\sigma_1$) or $0\neq x^{\frac{\lambda_1}{2}-1}v\in\im(x^{\lambda_1-1})$
(if $\rho_1=\sigma_1=\frac{\lambda_1}{2}$, forcing $k(\rho)=c$); so
whether $\lambda_1$ is even or odd, $W_{(v,x)}^\lambda=\im(x^{\lambda_1-1})$ and 
part (1) is proved.
If $(V_{\geq a})$ is a $\lambda$-filtration which is $C$-adapted to
$(v,x)$, then by definition $\im(x^{\lambda_1-1})=x^{\lambda_1-1}(V_{\geq 1-\lambda_1})$
is contained in $V_{\geq \lambda_1-1}$, and they are both $2c$-dimensional, proving (2).
For part (3),
one sees from explicit orbit representatives that
$(\rho';\sigma')$ is the bipartition obtained from the `bi-composition'
\[
(\rho_1-1,\cdots,\rho_c-1,\rho_{c+1},\rho_{c+2},\cdots;
\sigma_1-1,\cdots,\sigma_c-1,\sigma_{c+1},\sigma_{c+2},\cdots)
\]
by applying the rectification procedure of \cite[Lemma 2.4]{ah}. Namely, if
$k(\rho)=k(\sigma)=c$ then no modification is required. If
$k(\rho)=c$ and $k(\sigma)>c$, then we must subtract
one from each of $\sigma_{c+1},\cdots,\sigma_{k(\sigma)}$ and add one
to each of $\rho_{c+1},\cdots,\rho_{k(\sigma)}$ (and then re-order parts of the
first composition as required, 
if it happens that $\rho_{c+1}+1=\rho_1$). If $k(\rho)>c$ and $k(\sigma)=c$,
then we must
restore the original parts $\rho_1,\cdots,\rho_c$ and replace $\sigma_1-1,\cdots,
\sigma_c-1$ by $\sigma_1-2,\cdots,\sigma_c-2$ (and then re-order parts of the second
composition as required,
if it happens that $\sigma_{c+1}=\sigma_c-1$).
A short calculation shows that in any of these cases,
$\Phi^C(\rho';\sigma')=\lambda'$.

\noindent
\textbf{Case 3:} $k(\lambda)=2c+1$
for some $c\geq 1$. By the definition of $\Phi^C$, we must have $\lambda_1=2m$
where $m=\rho_1=\sigma_1$, and $k(\rho)>k(\sigma)=c$.
So $(\rho+\sigma)_1=2m$, and $k(\rho+\sigma)=c$.
It follows that $\im(x^{2m-1})$ is
$2c$-dimensional, contained in $\ker(x)$, and isotropic; 
also $x^{m-1}v$ is contained in $\ker(x)$, and not contained in $\im(x^{2m-1})$,
so $W_{(v,x)}^\lambda=\F x^{m-1}v+\im(x^{2m-1})$ is indeed 
$(2c+1)$-dimensional and part (1)
is proved. 
If $(V_{\geq a})$ is a $\lambda$-filtration which is $C$-adapted to
$(v,x)$, then by definition $W_{(v,x)}^\lambda$
is contained in $V_{\geq \lambda_1-1}$, and they are both $(2c+1)$-dimensional, 
proving (2).
In this case,
the passage from $(\rho;\sigma)$ to $(\rho';\sigma')$ can be done in two steps: first
finding the bipartition labelling the type of
$(v+\F x^{m-1}v,x|_{(\F x^{m-1}v)^\perp/\F x^{m-1}v})$ as in Case 1, then using that
as the starting point for the procedure in Case 2. Again it is easy to
see that $\Phi^C(\rho';\sigma')=\lambda'$.
\end{proof}

\begin{thm} \label{thm:filt-C}
Let $\lambda\in\cP_{2n}^C$.
For any $(v,x)\in\Tt_{\hPhi^C(\lambda)}^C$, there is a unique $\lambda$-filtration
$(V_{\geq a})$ of $V$ which is $C$-adapted to $(v,x)$.
\end{thm}
\begin{proof}
The proof is by induction on $n$: the $n=0$ case is vacuously true, so assume that
$n\geq 1$ and that the result is known for smaller values of $n$.
If $\lambda=(1^{2n})$, then 
$v=0$, $x=0$, and a $\lambda$-filtration involves no non-trivial subspaces, 
so the result is true. So we can assume that $\lambda_1\geq 2$.

By Proposition \ref{prop:recursion}, we have a $k(\lambda)$-dimensional
isotropic subspace $W_{(v,x)}^\lambda$ which is contained in $\ker(x)$,
and the induced pair $(v',x')=(v+W_{(v,x)}^\lambda,
x|_{(W_{(v,x)}^\lambda)^\perp/W_{(v,x)}^\lambda})$ lies in
$\Tt_{\hPhi^C(\lambda')}^C$. By the induction hypothesis applied to
$\lambda'$ and $(v',x')$, there exists a $\lambda'$-filtration $(V_{\geq a}')$ of 
$(W_{(v,x)}^\lambda)^\perp/W_{(v,x)}^\lambda$ which is $C$-adapted to 
$(v',x')$. Let
$(V_{\geq a})$ be the $\lambda$-filtration of $V$ obtained from
$W_{(v,x)}^\lambda$ and $(V_{\geq a}')$
by the procedure of
Proposition \ref{prop:filt}(2); we aim to show that this filtration is $C$-adapted
to $(v,x)$. 

From the fact that
$v+W_{(v,x)}^\lambda\in V_{\geq 1}'$ it follows that $v\in V_{\geq 1}$.
Trivially we have $x(V_{\geq a})\subseteq V_{\geq a+2}$ when $a\geq\lambda_1$ or
$a\leq -\lambda_1-1$.
By definition,
$V_{\geq \lambda_1-1}=W_{(v,x)}^\lambda$ and $V_{\geq 2-\lambda_1}=
(W_{(v,x)}^\lambda)^\perp$, so
$x(V_{\geq \lambda_1-1})\subseteq 
V_{\geq \lambda_1+1}=0$ and
$x(V_{\geq -\lambda_1})\subseteq V_{\geq 2-\lambda_1}$.
From the fact that $x'(V_{\geq a}')\subseteq V_{\geq a+2}'$ for all $a$ it follows
that $x(V_{\geq a})\subseteq V_{\geq a+2}$ when $2-\lambda_1\leq a\leq\lambda_1-3$.
All that remains to prove is
that $x(V_{\geq \lambda_1-2})\subseteq V_{\geq \lambda_1}=0$ (which implies
$x(V_{\geq 1-\lambda_1})\subseteq V_{\geq 3-\lambda_1}$). If $\lambda_1-1$ is not
a part of $\lambda$, then $V_{\geq \lambda_1-2}=V_{\geq \lambda_1-1}$.
So assume henceforth that $\lambda_1-1$ is a part of $\lambda$, which means that
$\lambda_1'=\lambda_1-1$.

If $\lambda_1=2$, then 
$V_{\geq \lambda_1-2}=(W_{(v,x)}^\lambda)^\perp=(\F v)^\perp\cap\ker(x)$,
so $x(V_{\geq \lambda_1-2})=0$. If $\lambda_1\geq 3$, then
$V_{\geq \lambda_1-2}'=W_{(v',x')}^{\lambda'}$ 
by Proposition \ref{prop:recursion}(2), so
\[
V_{\geq \lambda_1-2}=\begin{cases}
W_{(v,x)}^\lambda+\F x^{\frac{\lambda_1-1}{2}-1} v+ 
x^{\lambda_1-2}((W_{(v,x)}^\lambda)^\perp),
&\text{ if $\lambda_1$ is odd,}\\
W_{(v,x)}^\lambda+ x^{\lambda_1-2}((W_{(v,x)}^\lambda)^\perp)
,&\text{ if $\lambda_1$ is even.}
\end{cases}
\]
Since $W_{(v,x)}^\lambda\supseteq \im(x^{\lambda_1-1})$ by definition, we have
$(W_{(v,x)}^\lambda)^\perp\subseteq \ker(x^{\lambda_1-1})$.
Also, if $\lambda_1$ is odd, then the bipartition $(\rho;\sigma)$ such that
$(v,x)\in\Oo_{\rho;\sigma}$ must satisfy
$\rho_1\leq\frac{\lambda_1-1}{2}$, so
$x^{\frac{\lambda_1-1}{2}} v=0$. Thus $x(V_{\geq \lambda_1-2})=0$ as required.

To prove the uniqueness, suppose that $(U_{\geq a})$ is any $\lambda$-filtration
of $V$ which is $C$-adapted to $(v,x)$. Then
$U_{\geq \lambda_1-1}=W_{(v,x)}^\lambda$ by Proposition \ref{prop:recursion}(2).
The $\lambda'$-filtration $(U_{\geq a}')$ of
$(W_{(v,x)}^\lambda)^\perp/W_{(v,x)}^\lambda$ to which $(U_{\geq a})$ corresponds under
Proposition \ref{prop:filt} is clearly $C$-adapted to $(v',x')$, and hence must equal
the filtration $(V_{\geq a}')$ used above, by the uniqueness part of the induction 
hypothesis. So $(U_{\geq a})$ equals $(V_{\geq a})$.
\end{proof}

\begin{rmk}
Suppose $\F$ has characteristic $2$, and define $y=\Psi(v,x)\in\cN(\fsp(V))$ 
as in Section 3. If $(V_{\geq a})$ is $C$-adapted to $(v,x)$, then
$y(V_{\geq a})\subseteq V_{\geq a+2}$ for all $a$. So in this case
the unique $\lambda$-filtration $(V_{\geq a})$ is the one attached to the element 
$y\in\cT_{\hPhi^C(\lambda)}^C$ in \cite[A.3]{lusztig:unip}, 
for which a non-recursive formula
is given in \cite[2.3]{lusztig:unip1}.
\end{rmk}

Now let $(\mu;\nu)\in\cQ_n^C$ and set $\lambda=\Phi^C(\mu;\nu)$.
Fix a $\lambda$-filtration $(V_{\geq a})$ of $V$
and let $P$ be its stabilizer in $Sp(V)$.
Define
\[ S_{\geq 2}=\{x\in S\,|\,x(V_{\geq a})\subseteq V_{\geq a+2}\text{ for all }a\}. \]
Then we have a $P$-submodule $V_{\geq 1}\oplus S_{\geq 2}$ of $V\oplus S$, and an
associated vector bundle $Sp(V)\times^P(V_{\geq 1}\oplus S_{\geq 2})$ over $Sp(V)/P$.
We can identify $Sp(V)\times^P(V_{\geq 1}\oplus S_{\geq 2})$ with the variety
of triples $((U_{\geq a}),v,x)$ where $(v,x)\in\fN$ and
$(U_{\geq a})$ is a $\lambda$-filtration of 
$V$ which is $C$-adapted to $(v,x)$. We have a
proper map $\pi:Sp(V)\times^P(V_{\geq 1}\oplus S_{\geq 2})\to\fN$ which
forgets the filtration $(U_{\geq a})$.

\begin{thm} \label{thm:res-C}
With notation as above, 
the map $\pi:Sp(V)\times^P(V_{\geq 1}\oplus S_{\geq 2})\to\fN$
has image $\overline{\Oo_{\mu;\nu}}$ and is a resolution of that variety.
This resolution restricts to an isomorphism $Sp(V)\times^P U\isomto \Tt_{\mu;\nu}^C$
where $U$ is some open subvariety of $V_{\geq 1}\oplus S_{\geq 2}$. In particular,
$\Tt_{\mu;\nu}^C$ is smooth.
\end{thm}

\begin{proof}
We must first show that $Sp(V)\times^P(V_{\geq 1}\oplus S_{\geq 2})$ has the right
dimension, namely $2(n^2-b(\mu;\nu))$. The dimension is clearly independent
of characteristic, so we can assume temporarily that the characteristic is zero and that
$(V_{\geq a})$ arises from a Jacobson--Morozov grading $(V_a)$ associated to a
nilpotent $y\in\cN(\fsp(V))$ of Jordan type $\lambda$. Then $P$ is the canonical
parabolic attached to $y$, and we have a known resolution
$Sp(V)\times^P\fsp(V)_{\geq 2}\to\overline{\cO_{\mu;\nu}^C}$, where
\[
\fsp(V)_{\geq 2}=\{y\in\fsp(V)\,|\,y(V_{\geq a})\subseteq V_{\geq a+2}
\text{ for all }a\}.
\]
So it suffices to prove that $\dim \fsp(V)_{\geq 2}=\dim V_{\geq 1}+\dim S_{\geq 2}$,
which follows from the explicit formulas:
\begin{equation} \label{eqn:filtdim}
\begin{split}
\dim \fsp(V)_{\geq 2}&=\sum_{a\geq 1}\binom{\dim V_{a}+1}{2} +
\sum_{\substack{a\geq 2\\1-a\leq b\leq a-2}} (\dim V_a)(\dim V_b),\\
\dim S_{\geq 2}&=\sum_{a\geq 1}\binom{\dim V_{a}}{2} +
\sum_{\substack{a\geq 2\\1-a\leq b\leq a-2}} (\dim V_a)(\dim V_b).
\end{split}
\end{equation}

Since $\pi$ is proper and $Sp(V)\times^P(V_{\geq 1}\oplus S_{\geq 2})$ is irreducible,
the image of $\pi$ is an irreducible closed subvariety of $\fN$ of dimension
$\leq 2(n^2-b(\mu;\nu))$, which contains $\Tt_{\mu;\nu}^C$ by Theorem \ref{thm:filt-C},
and must therefore equal $\overline{\Oo_{\mu;\nu}}$. It is obvious that
$Sp(V)\times^P(V_{\geq 1}\oplus S_{\geq 2})$ is smooth. Its open subvariety
$\pi^{-1}(\Tt_{\mu;\nu}^C)$ is $Sp(V)$-stable, hence must be of the form
$Sp(V)\times^P U$ for some open subvariety $U$ of $V_{\geq 1}\oplus S_{\geq 2}$.
By the uniqueness in Theorem \ref{thm:filt-C}, the restriction
$\pi:Sp(V)\times^P U\to \Tt_{\mu;\nu}^C$ is bijective.

All that remains is to show that the inverse map $\Tt_{\mu;\nu}^C\to Sp(V)\times^P U$
is a morphism of varieties. This is equivalent to saying that the map which associates
to $(v,x)\in\Tt_{\mu;\nu}^C$ the unique $\lambda$-filtration 
specified in Theorem \ref{thm:filt-C}
is a morphism from $\Tt_{\mu;\nu}^C$ to $Sp(V)/P$. By the recursive construction used
to prove Theorem \ref{thm:filt-C}, this follows from the fact that the
map $(v,x)\mapsto W_{(v,x)}^\lambda$ 
is a morphism from $\Tt_{\mu;\nu}^C$ to the Grassmannian
of $k(\lambda)$-dimensional isotropic subspaces of $V$.
\end{proof}

\begin{rmk}
In the case $(\mu;\nu)=((n);\varnothing)$, a $\lambda$-filtration is a 
complete flag in $V$, and this resolution is nothing but Kato's
exotic Springer resolution of $\fN$ (see \cite{kato:exotic}).
\end{rmk}

The argument for the type-$B$ pieces is very similar, so we will concentrate on the points
which are different. For any $\lambda\in\cP_{2n+1}^B$, we define a
$\lambda$-filtration of $V$ by Definition \ref{defn:lambda-filt} again,
despite the fact that $|\lambda|=2n+1$ whereas $\dim V=2n$. Proposition
\ref{prop:filt} remains true for $\lambda\in\cP_{2n+1}^B$, except that now
$\lambda'\in\cP_{2(n-k(\lambda))+1}^B$. We must make a slight change to Definition
\ref{defn:C-adapted}:

\begin{defn} \label{defn:B-adapted}
For $(v,x)\in\fN$, we say that a $\lambda$-filtration $(V_{\geq a})$ is 
\emph{$B$-adapted} to $(v,x)$ if $v\in V_{\geq 2}$ and 
$x(V_{\geq a})\subseteq V_{\geq a+2}$ for all $a$.
\end{defn}

The analogue of Proposition \ref{prop:recursion} is as follows.

\begin{prop} \label{prop:recursion-B}
Let $\lambda\in\cP_{2n+1}^B$ with $\lambda_1\geq 2$, and define
$\lambda'$ as in Proposition \ref{prop:filt}.
Take any
$(v,x)\in\Tt_{\hPhi^B(\lambda)}^B$, and set
\[ W_{(v,x)}^\lambda=\begin{cases}
\im(x^{\lambda_1-1}),&\text{ if $\lambda_1$ is even,}\\
\F x^{\lfloor\frac{\lambda_1}{2}\rfloor-1}v+\im(x^{\lambda_1-1}),
&\text{ if $\lambda_1$ is odd.}
\end{cases} \]
\begin{enumerate}
\item $W_{(v,x)}^\lambda$ is $k(\lambda)$-dimensional, 
isotropic, and contained in $\ker(x)$.
\item If $(V_{\geq a})$ is a $\lambda$-filtration which is $B$-adapted to
$(v,x)$, then $V_{\geq\lambda_1-1}=W_{(v,x)}^\lambda$.
\item In the exotic nilpotent cone for the vector space
$(W_{(v,x)}^\lambda)^\perp/W_{(v,x)}^\lambda$, the induced pair 
$(v+W_{(v,x)}^\lambda,x|_{(W_{(v,x)}^\lambda)^\perp/W_{(v,x)}^\lambda})$ lies in
$\Tt_{\hPhi^B(\lambda')}^B$.
\end{enumerate}
\end{prop}

\begin{proof}
The proof is much the same as that of Proposition \ref{prop:recursion}.
Let $(\rho;\sigma)\in\cQ_n$ be such that $(v,x)\in\Oo_{\rho;\sigma}$;
by assumption, $\Phi^B(\rho;\sigma)=\lambda$. 
Let $(\rho';\sigma')\in\cQ_{n-k(\lambda)}$ be such that
$(v',x')\in\Oo_{\rho';\sigma'}$, where
$(v',x')=(v+W_{(v,x)}^\lambda,x|_{(W_{(v,x)}^\lambda)^\perp/W_{(v,x)}^\lambda})$;
then part (3) is equivalent to the statement that $\Phi^B(\rho';\sigma')=\lambda'$.

\noindent
\textbf{Case 1:} $k(\lambda)=1$. By the definition of
$\Phi^B$, we must have $\lambda_1=2m+1$ where $m=\rho_1\geq\sigma_1$.
Thus $(\rho+\sigma)_1\leq 2m$, which implies that $\im(x^{2m})=0$, so
$W_{(v,x)}^\lambda=\F x^{m-1}v$, and part (1) follows for the same reason as
in Proposition \ref{prop:recursion}. If $(V_{\geq a})$ is a $\lambda$-filtration 
which is $B$-adapted to $(v,x)$, then by definition we have
$x^{m-1}v\in V_{\geq 2m}$ and $\dim V_{\geq 2m}=1$, proving (2). The description
of $(\rho';\sigma')$ is the same as in Case 1 of Proposition \ref{prop:recursion},
and the proof that $\Phi^B(\rho';\sigma')=\lambda'$ is analogous.

\noindent
\textbf{Case 2:} $k(\lambda)=2c$ for some $c\geq 1$. 
By the definition of $\Phi^B$, we must have
$\rho_1+\sigma_1=\lambda_1$, and either $k(\rho)=c$, $k(\sigma)\geq c$, and
$\rho_1<\sigma_1$, or
$k(\rho)>c$, $k(\sigma)=c$, and $\rho_1<\sigma_1-1$.
Hence $(\rho+\sigma)_1=\lambda_1$, and $k(\rho+\sigma)=c$. 
It follows that $\im(x^{\lambda_1-1})$ is
$2c$-dimensional and contained in $\ker(x)$; also, it is isotropic, since its
perpendicular subspace is $\ker(x^{\lambda_1-1})$ and $\lambda_1\geq 2$.
If $\lambda_1$ is odd, then either $x^{\lfloor\frac{\lambda_1}{2}\rfloor-1}v=0$
(if $\rho_1<\sigma_1-1$) or $0\neq x^{\lfloor\frac{\lambda_1}{2}\rfloor-1}v
\in\im(x^{\lambda_1-1})$
(if $\rho_1=\sigma_1-1=\lfloor\frac{\lambda_1}{2}\rfloor$, forcing $k(\rho)=c$); so
whether $\lambda_1$ is even or odd, $W_{(v,x)}^\lambda=\im(x^{\lambda_1-1})$ and 
part (1) is proved. The rest of the case is almost identical to Case 2 of
Proposition \ref{prop:recursion}.

\noindent
\textbf{Case 3:} $k(\lambda)=2c+1$ for some $c\geq 1$. 
By the definition of $\Phi^B$, we must have $\lambda_1=2m+1$
where $m=\rho_1=\sigma_1-1$, and $k(\rho)>k(\sigma)=c$.
So $(\rho+\sigma)_1=2m+1$, and $k(\rho+\sigma)=c$.
It follows that $\im(x^{2m})$ is
$2c$-dimensional, contained in $\ker(x)$, and isotropic; 
also $x^{m-1}v$ is contained in $\ker(x)$, and not contained in $\im(x^{2m})$,
so $W_{(v,x)}^\lambda=\F x^{m-1}v+\im(x^{2m})$ is indeed 
$(2c+1)$-dimensional and part (1) is proved. 
The rest proceeds as in Case 3 of
Proposition \ref{prop:recursion}.
\end{proof}

\begin{thm} \label{thm:filt-B}
Let $\lambda\in\cP_{2n+1}^B$.
For any $(v,x)\in\Tt_{\hPhi^B(\lambda)}^B$, there is a unique $\lambda$-filtration
$(V_{\geq a})$ of $V$ which is $B$-adapted to $(v,x)$.
\end{thm}
\begin{proof}
This is entirely analogous to the proof of Theorem \ref{thm:filt-C}, using
Proposition \ref{prop:recursion-B} in place of Proposition \ref{prop:recursion}.
\end{proof}

Now let $(\mu;\nu)\in\cQ_n^B$ and set $\lambda=\Phi^B(\mu;\nu)$. Fix a
$\lambda$-filtration $(V_{\geq a})$ of $V$ and let $P$ be its stabilizer in $Sp(V)$.
Then $Sp(V)\times^P(V_{\geq 2}\oplus S_{\geq 2})$ can be identified with the variety
of triples $((U_{\geq a}),v,x)$ where $(v,x)\in\fN$ and $(U_{\geq a})$ is a 
$\lambda$-filtration of $V$ which is $B$-adapted to $(v,x)$. 
We have a proper map
$\pi:Sp(V)\times^P(V_{\geq 2}\oplus S_{\geq 2})\to \fN$ 
which forgets the filtration $(U_{\geq a})$.

\begin{thm} \label{thm:res-B}
With notation as above, the map $\pi:Sp(V)\times^P(V_{\geq 2}\oplus S_{\geq 2})\to \fN$
has image $\overline{\Oo_{\mu;\nu}}$ and is a resolution of that variety.
This resolution restricts to an isomorphism $Sp(V)\times^P U\isomto \Tt_{\mu;\nu}^B$
where $U$ is some open subvariety of $V_{\geq 2}\oplus S_{\geq 2}$. In particular,
$\Tt_{\mu;\nu}^B$ is smooth.
\end{thm}
\begin{proof}
The only part which is different from the proof of Theorem \ref{thm:res-C} is the
verification that $\dim Sp(V)\times^P(V_{\geq 2}\oplus S_{\geq 2})=2(n^2-b(\mu;\nu))$, 
for which
we can assume that the characteristic is zero. Let $\tV$ be a $(2n+1)$-dimensional
vector space with a nondegenerate quadratic form, and let $(\tV_a)$ be a Jacobson--Morozov
grading of $\tV$ associated to a nilpotent $y\in\cN(\fo(\tV))$ of Jordan type $\lambda$.
Define a filtration $\tV_{\geq a}=\bigoplus_{b\geq a}\tV_b$ and let $\widetilde{P}$
be its stabilizer in $SO(\tV)$, a parabolic subgroup. We have a known resolution
$SO(\tV)\times^{\widetilde{P}}\fo(\tV)_{\geq 2}\to\overline{\cO_{\mu;\nu}^B}$, so it
suffices to show that
\begin{equation}
\dim(SO(\tV)/\widetilde{P})+\dim\fo(\tV)_{\geq 2}=\dim(Sp(V)/P)+\dim V_{\geq 2}+
\dim S_{\geq 2}.
\end{equation}
Now the fact that $(V_{\geq a})$ is a $\lambda$-filtration implies that
\[ 
\dim \tV_{\geq a}=\begin{cases}
\dim V_{\geq a}&\text{ if $a\geq 1$,}\\
\dim V_{\geq a}+1&\text{ if $a\leq 0$.}
\end{cases}
\]
It follows easily that the two partial flag varieties $SO(\tV)/\widetilde{P}$ and
$Sp(V)/P$ have the same dimension. Moreover,
\begin{equation}
\dim\fo(\tV)_{\geq 2}=\sum_{a\geq 1}\binom{\dim \tV_{a}}{2} +
\sum_{\substack{a\geq 2\\1-a\leq b\leq a-2}} (\dim \tV_a)(\dim \tV_b),
\end{equation}
and comparing with \eqref{eqn:filtdim} gives the result.
\end{proof}

Now suppose that $(\mu;\nu)\in\cQ_n^\circ$.
In characteristic $\neq 2$, Kraft and Procesi showed in \cite{kp:special}
that the special pieces 
$\cS^B_{\mu;\nu}$ and $\cS^C_{\mu;\nu}$ are rationally smooth.  
This can be expressed as a statement about stalks of intersection cohomology complexes on these varieties: specifically, that
$\dim \IC(\overline{\cO^B_{\mu;\nu}})|_x = 1$
for all $x \in \cS^B_{\mu;\nu}$,
and similarly for $\cS^C_{\mu;\nu}$.  (This fact, conjectured by 
Lusztig in \cite{lusztig:green}, is now known to hold for all simple Lie algebras 
in good characteristic.) 

\begin{conj}
For any $(\mu;\nu)\in\cQ_n^\circ$, the special pieces $\Ss_{\mu;\nu}$ of $\fN$ in any
characteristic, and the special pieces $\cS^B_{\mu;\nu}$ of $\cN(\fo_{2n+1})$
and $\cS^C_{\mu;\nu}$ of $\cN(\fsp_{2n})$ in characteristic $2$, are also
rationally smooth.
\end{conj}

Note that Theorems \ref{thm:res-C} and \ref{thm:res-B} give two (usually different)
resolutions of $\overline{\Oo_{\mu;\nu}}$, and show that the open subvariety 
$\Tt_{\mu;\nu}^C\cup\Tt_{\mu;\nu}^B$ of $\Ss_{\mu;\nu}$ is smooth.

\begin{exam} \label{exam:maybenotsmoothspecial}
The smallest $\Ss_{\mu;\nu}$ which is not proved to be smooth by this argument is
the $58$-dimensional $\Ss_{(21);(21)}$ for $n=6$. We have
\[
\begin{split}
\Tt_{(21);(21)}^C&=\Oo_{(21);(21)}\disjunion\Oo_{(1^2);(31)}\disjunion\Oo_{(2);(2^2)}
\disjunion\Oo_{(1);(32)}\disjunion\Oo_{\varnothing;(42)},\\
\Tt_{(21);(21)}^B&=\Oo_{(21);(21)}\disjunion\Oo_{(2^2);(1^2)}\disjunion\Oo_{(21^2);(2)}
\disjunion\Oo_{(2^2 1);(1)},\\
\Ss_{(21);(21)}&=\Tt_{(21);(21)}^C\cup\Tt_{(21);(21)}^B\cup\Oo_{(1^3);(3)}.
\end{split}
\]
The resolutions of $\overline{\Oo_{(21);(21)}}$ given by 
Theorems \ref{thm:res-C} and \ref{thm:res-B} both have $1$-dimensional fibres over
the orbit $\Oo_{(1^3);(3)}$. Using the Decomposition Theorem, one can check that
$\Ss_{(21);(21)}$ is rationally smooth.
\end{exam}

The rational smoothness of $\cS^C_{\mu;\nu}$ in characteristic $2$ is equivalent to that
of $\Ss_{\mu;\nu}$, by the following result:

\begin{lem}\label{lem:fnsp-stalk}
Suppose $\F$ has characteristic $2$. 
For any $(\mu;\nu) \in \cQ_n$, the $Sp(V)$-equivariant bijection 
$\Psi: \fN \to \cN(\fsp(V))$ induces isomorphisms of stalks
\[
\IC(\overline{\Oo_{\mu;\nu}})|_{(v,x)} \cong 
\IC(\overline{\cO^{C,2}_{\mu;\nu}})|_{\Psi(v,x)}
\]
for all $(v,x) \in \overline{\Oo_{\mu;\nu}}$.
\end{lem}
\begin{proof}
Because $\Psi$ is a finite morphism, the functor $R\Psi_*$ is exact on the category of perverse sheaves.  Moreover, $R\Psi_*\IC(\overline{\Oo_{\mu;\nu}})$ is the intersection cohomology complex associated to the local system $\mathcal{L}$ on $\cO^{C,2}_{\mu;\nu}$ obtained by applying $\Psi_*$ to the constant sheaf on $\Oo_{\mu;\nu}$.  
Because $\Psi$ is a bijection, $\mathcal{L}$ must also be a constant sheaf.  In other words, $R\Psi_*\IC(\overline{\Oo_{\mu;\nu}}) \cong \IC(\overline{\cO^{C,2}_{\mu;\nu}})$.  
Finally, the proper base change theorem yields an isomorphism
$\IC(\overline{\Oo_{\mu;\nu}})|_{(v,x)} \cong 
R\Psi_*\IC(\overline{\Oo_{\mu;\nu}})|_{\Psi(v,x)}$,
completing the proof.
\end{proof}

We end with some remarks on the problem of computing stalks of intersection cohomology complexes of orbit closures in general.
For nilpotent orbits in good characteristic, these stalks can be computed by an algorithm described by Lusztig~\cite[\S24]{lusztig:cs5}, following earlier work of Shoji~\cite{shoji:green}.  Several variants of this algorithm have appeared in the literature, including one introduced by Shoji~\cite{shoji:limit} related to the combinatorics 
of \emph{limit symbols}.  In~\cite{ah}, it was shown that Shoji's limit-symbol version of the algorithm computes stalks of intersection cohomology complexes on the enhanced nilpotent cone.  It was conjectured in~\cite[Conjecture~6.4]{ah} that the same 
algorithm works for $\fN$ as well; further support for this conjecture comes from
Kato's work on the exotic Springer correspondence in \cite{kato:deformations}.

In view of Lemma \ref{lem:fnsp-stalk} and the work of 
Xue in \cite{xue,xue:combinatorics}, we have a new interpretation of
\cite[Conjecture~6.4]{ah}: it asserts that the 
stalks of intersection cohomology complexes for orbit closures in $\cN(\fsp(V))$ in
characteristic $2$ are given by the Lusztig--Shoji algorithm, defined 
using the Springer correspondence in characteristic $2$.
It is natural to guess that this holds for $\cN(\fo(\tV))$ also.


\end{document}